\documentclass{amsart}
\usepackage[all]{xy}
\usepackage{color}

\newtheorem{theorem}{Theorem}[section]
\newtheorem{lemma}[theorem]{Lemma}
\newtheorem{proposition}[theorem]{Proposition}
\newtheorem{corollary}[theorem]{Corollary}

\newtheorem{definition}[theorem]{Definition}
\newtheorem{example}[theorem]{Example}

\theoremstyle{remark}
\newtheorem{notation}[theorem]{Notation}

\theoremstyle{remark}
\newtheorem{remark}[theorem]{Remark}

\newcommand{\ca}{{\mathcal A}}

\newcommand{\cm}{{\mathcal M}}
\newcommand{\cn}{{\mathcal N}}
\newcommand{\cb}{{\mathcal B}}

\newcommand{\cc}{{\mathcal C}}

\newcommand{\cg}{{\mathcal G}}
\newcommand{\co}{{\mathcal O}}
\newcommand{\ch}{{\mathcal H}}
\newcommand{\cd}{{\mathcal D}}
\newcommand{\ci}{{\mathcal I}}

\newcommand{\cR}{{\mathcal R}}
\newcommand{\cs}{{\mathcal S}}
\newcommand{\cu}{{\mathcal U}}
\newcommand{\cv}{{\mathcal V}}
\newcommand{\cw}{{\mathcal W}}
\newcommand{\cz}{{\mathcal Z}}
\newcommand{\dgcat}{\mathsf{dgcat}}
\newcommand{\ccat}{\cc \text{-}\mbox{Cat}}
\newcommand{\spc}{\mathsf{Sp}^{\Sigma}(\cc,K)}
\newcommand{\spcc}{\mathsf{Sp}^{\Sigma}(\cc,K)\text{-}\mbox{Cat}}

\newcommand{\internalcomment}[1]{}

\begin{document}

\title[THH and TC for dg categories]{Topological Hochschild and cyclic homology for differential graded categories}
\author{Gon{\c c}alo Tabuada}
\address{Departamento de Matematica, FCT-UNL, Quinta da Torre, 2829-516 Caparica,~Portugal}


\email{
\begin{minipage}[t]{5cm}
tabuada@fct.unl.pt
\end{minipage}
}

\begin{abstract}
We define a topological Hochschild ($THH$) and cyclic ($TC$) homology theory for differential graded (dg) categories and construct several non-trivial natural transformations from algebraic $K$-theory to $THH(-)$.

In an intermediate step, we prove that the homotopy theory of dg categories is Quillen equivalent, through a four step zig-zag of Quillen equivalences, to the homotopy theory of Eilenberg-Mac Lane spectral categories.

Finally, we show that over the rationals $\mathbb{Q}$ two dg categories are topological equivalent if and only if they are quasi-equivalent.
\end{abstract}

\maketitle

\tableofcontents

\section{Introduction}
In the past two decades, topological Hochschild homology ($THH$) and topological cyclic homology ($TC$) have revolutionized algebraic $K$-theory computations \cite{Madsen} \cite{Madsen1} \cite{Madsen2}. Roughly, the $THH$ of a ring spectrum is obtained by substituting the tensor product in the classical Hochschild complex by the smash product. The $THH$ spectrum comes with a `cyclotomic' structure, and for each prime $p$, $TC$ is defined as a certain homotopy limit over the fixed point spectra.

The aim of this article is to twofold: (1) to defined these new homology theories in the context of {\em differential graded} (dg) categories (over a base ring $R$) \cite{Drinfeld} \cite{ICM} \cite{These}; (2) to relate these new theories with algebraic $K$-theory. Our motivation comes from Drinfeld-Kontsevich's program in non-commutative algebraic geometry \cite{Chitalk} \cite{ENS} \cite{finMotiv}, where dg categories are considered as the correct models for non-commutative spaces. For example, one can associate to a smooth projective algebraic variety $X$ a dg model $\cd^b_{dg}(\mbox{coh}(X))$, i.e. a dg category well defined up to quasi-equivalence (see~\ref{we}), whose triangulated category $H^0(\cd^b_{dg}(\mbox{coh}(X)))$ (obtained by applying the zero cohomology group functor in each Hom-complex) is equivalent to the bounded derived category of quasi-coherent complexes on $X$. For $\cd^b_{dg}(\mbox{coh}(X))$, one could take  the dg category of left bounded complexes of injective $\co_X$-modules whose cohomology is bounded and coherent.

As the above example shows, these theories would furnishes us automatically a well-defined topological Hochschild and cyclic homology theory for algebraic varieties.

In order to define these new homology theories for dg categories, we introduce in chapter~\ref{chapter7} the notion of {\em Eilenberg-Mac Lane spectral category} (i.e. a category enriched in modules over the Eilenberg-Mac Lane ring spectrum $HR$), and prove (by inspiring ourselves in Shipley's work~\cite{Shipley}) our first main theorem.

\begin{theorem}{(see chapter~\ref{global})}\label{M1}
The homotopy theory of differential graded categories is Quillen equivalent, through a four step zig-zag of Quillen equivalences, to the homotopy theory of Eilenberg-Mac Lane spectral categories.
\end{theorem}
The novelty in these homotopy theories is that in contrast to dg algebras or enriched categories with a fixed set of objects, the correct notion of equivalence (as the above example shows) is not defined by simply forgetting the multiplicative structure. It is an hybrid notion which involves weak equivalences and categorical equivalences (see definition \ref{we}), making not only all the arguments very different but also much more evolved. In order to prove theorem~\ref{M1}, we had to develop several new general homotopical algebra tools (see theorems \ref{stable}, \ref{conditionS}, \ref{chave} and \ref{stable1}) which are of independent interest.

By considering the restriction functor
$$ HR\text{-}\mbox{Cat} \longrightarrow  \mathsf{Sp}^{\Sigma}\text{-}\mbox{Cat}$$
from Eilenberg-Mac Lane to spectral categories (i.e. categories enriched over symmetric spectra), we obtain a well-defined functor
$$\mathsf{Ho}(\dgcat) \stackrel{\sim}{\rightarrow} \mathsf{Ho}(HR\text{-}\mbox{Cat}) \rightarrow \mathsf{Ho}(\mathsf{Sp}^{\Sigma}\text{-}\mbox{Cat})$$
on the homotopy categories. Finally, by composing the above functor with the topological and cyclic homology theories defined by Blumberg and Mandell~\cite{Mandell} for spectral categories, we obtain our searched homology theories
$$ THH, TC:\mathsf{Ho}(\dgcat) \longrightarrow \mathsf{Ho}(\mathsf{Sp})\,.$$

Calculations in algebraic $K$-theory are very rare and hard to get by. Our second main theorem establishes a relationship between algebraic $K$-theory and topological Hochschild homology.
\begin{theorem}[see theorem~\ref{natural}]\label{M2}
Let $R$ be the ring of integers $\mathbb{Z}$. Then we have non-trivial natural transformations
$$
\begin{array}{rcl}
 \gamma_n: K_n(-) & \Rightarrow & THH_n(-)\,,\,\,\,\,\,\, n \geq 0\\
 \gamma_{n,r}: K_n(-) & \Rightarrow & THH_{n+2r-1}(-)\,,\,\,\,\,\,\, n\geq 0\,,\ r\geq 1\,.
\end{array}
$$
from the algebraic $K$-theory groups to the topological Hochschild homology ones.
\end{theorem}
When $X$ is a quasi-compact and quasi-separated scheme, the $THH$ of $X$ is equivalent \cite[1.3]{Mandell} to the topological Hochschild homology of schemes as defined by Geisser and Hesselholt in \cite{GH}. In particular, theorem~\ref{M2} allows us to use all the calculations of $THH$ developed in \cite{GH}, to infer results on algebraic $K$-theory.

Finally, we recall in chapter \ref{chapter9} Dugger-Shipley's notion of topological equivalence~\cite{DS} and show in proposition~\ref{proprational} that, when $R$ is the field of rationals numbers $\mathbb{Q}$, two dg categories are topological equivalent if and only if they are quasi-equivalent.

\section{Acknowledgements}
It is a great pleasure to thank Stefan Schwede for suggesting me this problem and Gustavo Granja for helpful conversations. I would like also to thank the Institute for Advanced Study at Princeton for excellent working conditions, where part of this work was carried out.

\section{Preliminaries}
{\bf Throughout this article we will be working (to simplify notation) over the integers $\mathbb{Z}$. However all our results are valid if we replace $\mathbb{Z}$ by any unital commutative ring $R$.}

\vspace{0.1cm}

Let $Ch$ denote the category of complexes and $Ch_{\geq 0}$ the full subcategory of positive graded complexes.  We consider homological notation (the differential decreases the degree). We denote by $s\mathbf{Ab}$ the category of simplicial abelian groups.

\begin{notation}\label{catenrich}
Let $(\cc, -\otimes_{\cc}-, {\bf 1}_{\cc})$ be a symmetric monoidal category \cite[6.1.2]{Borceaux}. We denote by $\cc \text{-}\mbox{Cat}$ the category of small categories enriched over $\cc$, see \cite[6.2.1]{Borceaux}. An object in $\cc \text{-} \mbox{Cat}$ will be called a {\em $\cc$-category} and a morphism a {\em $\cc$-functor}.
\end{notation}
\begin{notation}\label{catenrich1}
Let $\cd$ be a Quillen model category, see \cite{Quillen}. We denote by $\cw_{\cd}$ the class of weak equivalences in $\cd$ and by $[-,-]_{\cd}$ the set of morphisms in the homotopy category $\mathsf{Ho}(\cd)$ of $\cd$. Moreover, if $\cd$ is cofibrantly generated, we denote by $I_{\cd}$, resp. by $J_{\cd}$, the set of generating cofibrations, resp. generating trivial cofibrations.
\end{notation}
\begin{definition}\label{lax}
Let $(\cc, -\otimes-, {\bf 1}_{\cc})$ and $(\cd, -\wedge-, {\bf 1}_{\cd})$ be two symmetric monoidal categories. A {\em lax monoidal functor} is a functor $F:\cc \rightarrow \cd$ equipped with:
\begin{itemize}
\item[-] a morphism $\eta: {\bf 1}_{\cd} \rightarrow F({\bf 1}_{\cc})$ and 
\item[-] natural morphisms
$$ \psi_{X,Y}: F(X) \wedge F(Y) \rightarrow F(X \otimes Y),\,\,\,\, X,Y \in \cc$$
which are coherently associative and unital (see diagrams $6.27$ and $6.28$ in \cite{Borceaux}).
\end{itemize}
A lax monoidal functor is {\em strong monoidal} if the morphisms $\eta$ and $\psi_{X,Y}$ are isomorphisms.
\end{definition}
Now, suppose that we have an adjunction
$$
\xymatrix{
\cc \ar@<1ex>[d]^N \\
\cd \ar@<1ex>[u]^L\,,
}
$$
with $N$ a lax monoidal functor. Then, recall from~\cite[3.3]{SS}, that the left adjoint $L$ is endowed with a lax {\em co}monoidal structure.
$$
\begin{array}{lccc}
\gamma : & L({\bf 1}_{\cd}) & \longrightarrow &  {\bf 1}_{\cc}\\
\phi: & L(X\wedge Y) & \longrightarrow & L(X)\otimes L(Y)\,.
\end{array}
$$
\begin{definition}
A {\em weak monoidal Quillen pair} 
$$
\xymatrix{
\cc \ar@<1ex>[d]^N \\
\cd \ar@<1ex>[u]^L\,,
}
$$
between monoidal model categories $\cc$ and $\cd$ (see \cite[4.2.6]{HoveyLivro}) consists of a Quillen pair, with a lax monoidal structure on the right adjoint, such that the following two conditions hold:
\begin{itemize}
\item[(i)] for all cofibrant objects $X$ and $Y$ in $\cd$, the {\em co}monoidal map
$$ \phi: L(X\wedge Y) \longrightarrow L(X)\otimes L(Y)$$
belongs to $\cw_{\cc}$ and
\item[(ii)] for some (hence any) cofibrant resolution ${\bf 1}_{\cd}^c \stackrel{\sim}{\rightarrow} {\bf 1}_{\cd}$, the composite map
$$ L({\bf 1}_{\cd}^c) \rightarrow L({\bf 1}_{\cd}) \stackrel{\gamma}{\rightarrow} {\bf 1}_{\cc}$$
belongs to $\cw_{\cc}$.
\end{itemize}
A {\em strong monoidal Quillen pair} is a weak monoidal Quillen pair for which the {\em co}monoidal maps $\gamma$ and $\phi$ are isomorphisms.
\end{definition}
\begin{remark}
If ${\bf 1}_{\cd}$ is cofibrant and $L$ is strong monoidal, then $N$ is lax monoidal and the Quillen pair is a  strong monoidal Quillen pair.  
\end{remark}

\begin{definition}\label{smwe}
A weak (resp. strong) monoidal Quillen pair is a {\em weak monoidal Quillen equivalence} (respectively {\em strong monoidal Quillen equivalence}) if the underlying pair is a Quillen equivalence.
\end{definition}

Let $\cc$ a symmetric monoidal model category~\cite[4.2.6]{HoveyLivro}. We consider the following natural notion of equivalence in $\ccat$:
\begin{remark}\label{[]}
We have a functor
$$ [-]: \ccat \longrightarrow \mbox{Cat}\,,$$
with values in the category of small categories, obtained by base change along the monoidal composite functor
$$ \cc \longrightarrow \mathsf{Ho}(\cc) \stackrel{[{\bf 1}_{\cc},-]}{\longrightarrow} \mathbf{Set}\,.$$
\end{remark}
\begin{definition}\label{we}
A $\cc$-functor $F:\ca \rightarrow \cb$ is a {\em weak equivalence} if:
\begin{itemize}
\item[WE1)] for all objects $x,y \in \ca$, the morphism 
$$ F(x,y): \ca(x,y) \rightarrow \cb(Fx,Fy)$$
belongs to $\cw_{\cc}$ and 
\item[WE2)] the induced functor
$$ [F]: [\ca] \longrightarrow [\cb] $$
is an equivalence of categories.
\end{itemize}
\end{definition}
\begin{remark}
If condition $WE1)$ is verified, condition $WE2)$ is equivalent to:
\begin{itemize}
\item[WE2')] the induced functor
$$ [F]:[\ca] \longrightarrow [\cb]$$
is essentially surjective.
\end{itemize}
\end{remark}
\begin{definition}
Let $\ca \in \ccat$. A {\em right $\ca$-module} is a $\cc$-functor from $\ca$ to $\cc$.
\end{definition}
\begin{remark}\label{modules}
If the model category $\cc$ is cofibrantly generated, the category $\ca\text{-}\mbox{Mod}$ of right $\ca$-modules is naturally endowed with a (cofibrantly generated) Quillen model structure, see for example~\cite[11]{Hirschhorn}. We denote by $\mathsf{Ho}(\ca\text{-}\mbox{Mod})$ its homotopy triangulated category.
\end{remark}
We finish this chapter with some useful criterions:
\begin{proposition}{\cite[A.3]{Dugger}}\label{crit}
Let 
$$
\xymatrix{
\cc \ar@<1ex>[d]^U \\
\cd \ar@<1ex>[u]^F
}
$$
be an adjunction between Quillen model categories. If the right adjoint functor $U$ preserves trivial fibrations and fibrations between fibrant objects, then the adjunction $(F,U)$ is a Quillen adjunction.
\end{proposition}

\begin{proposition}{\cite[1.3.16]{HoveyLivro}}\label{Ncrit}
Let 
$$
\xymatrix{
\cc \ar@<1ex>[d]^U \\
\cd \ar@<1ex>[u]^F
}
$$
be a Quillen adjunction between model categories. The adjunction $(F,U)$ is a Quillen equivalence if and only if:
\begin{itemize}
\item[a)] the right adjoint $U$ reflects weak equivalences between fibrant objects and
\item[b)] for every cofibrant object $X \in \cd$, the composed morphism
$$
\xymatrix{
X \ar[dr] \ar[rr]^{\sim} && UF(X)_f \\
& UF(X) \ar[ur]_{U(i)} & \,,
}
$$
is a weak equivalence in $\cd$, where $i:F(X) \rightarrow F(X)_f$ is a fibrant resolution in $\cc$.
\end{itemize}
\end{proposition}
Throughout this article the adjunctions are displayed vertically with the left, resp. right, adjoint on the left side, resp. right side.
\section{Homotopy theory of general Spectral categories}\label{HomSpc}
Hovey developed in \cite{Hovey} the homotopy theory of symmetric spectra in general Quillen model categories. Starting from a `well-behaved' monoidal model category $\cc$ and a cofibrant object $K \in \cc$, Hovey constructed the stable monoidal model category $\spc$ of $K$-symmetric spectra over $\cc$, see \cite[7.2]{Hovey}.

In this chapter, we construct the homotopy theory of {\em general Spectral categories}, i.e. small categories enriched over $\mathsf{Sp}^{\Sigma}(\cc,K)$ (see theorem~\ref{stable}). For this, we impose the existence of a `well-behaved' Quillen model structure on $\ccat$ and certain conditions on $\cc$ and $K$.

Let $(\cc, -\otimes_{\cc}-, {\bf 1}_{\cc})$ be a cofibrantly generated monoidal model category and $K$ an object of $\cc$.
\subsection{Conditions on $\cc$ and $K$:}
\begin{itemize}
\item[C1)] The model structure on $\cc$ is proper (\cite[13.1.1]{Hirschhorn}) and cellular (\cite[12.1.1]{Hirschhorn}).
\item[C2)] The domains of the morphisms of the sets $I_{\cc}$ (generating cofibrations) and $J_{\cc}$ (generating trivial cofibrations) are cofibrant and sequentially small.
\item[C3)] The model structure on $\cc$ satisfies the monoid axiom (\cite[3.3]{SS}).
\item[C4)] There exists a product preserving functor $|-|:\cc \rightarrow \mathbf{Set}$ such that a morphism $f$ belongs to $\cw_{\cc}$ (see \ref{catenrich1}) if and only if $|f|$ is an isomorphism in $\mathbf{Set}$.
\item[C5)] The object $K \in \cc$ is cofibrant and sequentially small.
\item[C6)] The stable model structure on $\spc$ (see~\cite[8.7]{Hovey}) is right proper and satisfies the monoid axiom.
\item[C7)] The identity ${\bf 1}_{\cc}$ is a domain or codomain of an element of $I_{\cc}$.
\end{itemize}
\begin{remark}\label{rk1}
\begin{itemize}
\item[-] Notice that since $\cc$ is right proper, so it is the level model structure on $\spc$ (see~\cite[8.3]{Hovey}).
\item[-] Notice also, that since the evaluation functors (\cite[7.3]{Hovey})
$$ Ev_n: \spc \longrightarrow \cc \,,\,\,\, n\geq 0$$
preserve filtered colimits, the domains of the generating (trivial) cofibrations of the level model structure on $\spc$ are also sequentially small.
\item[-] Observe that the functor $|-|$ admits a natural extension:
$$ 
\begin{array}{rcl}
|-|_T: \spc & \longrightarrow &\underset{n \geq 0}{\prod} \mathbf{Set} \\
(X_n)_{n \geq 0} & \mapsto & (|X_n|)_{n \geq 0}\,.
\end{array}
$$
In this way a morphism $f$ in $\spc$ is a level equivalence if and only if $|f|_T$ is an isomorphism.
\end{itemize}
\end{remark}
\begin{example}\label{exampleC}
\begin{itemize}
\item[1)] If we consider for $\cc$ the projective model structure on the category $Ch$ (see \cite[2.3.11]{HoveyLivro}) and for $K$ the chain complex $\mathbb{Z}[1]$, which contains a single copy of $\mathbb{Z}$ in dimension $1$, then conditions $C1)\text{-}C7)$ are satisfied: conditions $C1)$, $C2)$, $C5)$ and $C7)$ are verified by construction. For condition $C4)$, consider the functor, which associates to a chain complex $M$ the set $\underset{n \in \mathbb{Z}}{\prod}H_n(M)$. Finally condition $C3)$ is proved in \cite[3.1]{Shipley} and condition $C6)$ is proved in \cite[2.9]{Shipley}.

\item[2)] If we consider for $\cc$ the projective model structure on $Ch_{\geq 0}$ (see \cite[III]{Jardine}) and for $K$ the chain complex $\mathbb{Z}[1]$, then as in the previous example, conditions $C1)\text{-}C7)$ are satisfied.

\item[3)] If we consider for $\cc$ the model structure on $s\mathbf{Ab}$ defined in \cite[III-2.8]{Jardine} and for $K$ the simplicial abelian group $\widetilde{\mathbb{Z}}(S^1)$, where $S^1=\Delta[1]/\partial \Delta[1]$ and $\widetilde{\mathbb{Z}}(S^1)_n$ is the free abelian group on the non-basepoint $n$-simpleces in $S^1$, then conditions $C1)\text{-}C7)$ are satisfied: conditions $C1)$, $C2)$, $C5)$ and $C7)$ are verified by construction. For condition $C4)$, consider the functor which associates to a simplicial abelian group $M$, the set $\underset{n \in \mathbb{N}}{\prod} \pi_n(M)$. Finally conditions $C3)$ and $C6)$ are proved in \cite[3.4]{Shipley}.
\end{itemize}
\end{example}
\subsection{Conditions on $\ccat$:}
We impose the {\em existence} of a (fixed) set $J''$ of $\cc$-functors, such that: the category $\ccat$ carries a cofibrantly generated Quillen model, whose weak equivalences are those defined in \ref{we} and the sets of generating (trivial) cofibrations are as follows:
\begin{definition}
The set $I$ of {\em generating cofibrations} on $\ccat$ consist of:
\begin{itemize}
\item[-] the $\cc$-functors obtained by applying the functor $U$ (see \ref{funcU}) to the set $I_{\cc}$ of generating cofibrations on $\cc$ and
\item[-] the $\cc$-functor
$$ \emptyset \longrightarrow \underline{{\bf 1}_{\cc}} $$
from the empty $\cc$-category $\emptyset$ (which is the initial object in $\cc\text{-}\mbox{Cat}$) to the $\cc$-category $\underline{{\bf 1}_{\cc}}$ with one object $\ast$ and endomorphism ring ${\bf 1}_{\cc}$.
\end{itemize}
\end{definition}

\begin{definition}\label{gencofCcat}
The set $J:=J'\cup J''$ of {\em generating trivial cofibrations} on $\ccat$ consist of:
\begin{itemize}
\item[-] the set $J'$ of $\cc$-functors obtained by applying the functor $U$ to the set $J_{\cc}$ of generating trivial cofibrations on $\cc$ and
\item[-] the fixed set $J''$ of $\cc$-functors.
\end{itemize}
\end{definition}
Moreover, we {\em impose} that:
\begin{itemize}
\item[H1)] the $\cc$-functors in $J''$ are of the form $ \cg \stackrel{V}{\longrightarrow} \cd $, with $\cg$ a sequentially small $\cc$-category with only one object $\{1\}$ and $\cd$ a $\cc$-category with two objects $\{1,2\}$, such that the map $\xymatrix{*+<1pc>{\cg(1)} \ar@{>->}[r]^{V(1)}_{\sim} & \cd(1)}$ is a trivial cofibration in $\cc^{\{1\}}\text{-}\mbox{Cat}$, see remark~\ref{nova1}.

\item[H2)] for every $\cc$-category $\ca$, the (solid) diagram
$$
\xymatrix{
\cb_1 \ar[r] \ar[d]_{P''} & \ca \\
\cb_2 \ar@{-->}[ur]_{\phi}
}
$$
admits a `lift' $\phi$, where $P''$ is a $\cc$-functor in $J''$.
\item[H3)] the Quillen model structure obtained is right proper.
\end{itemize}
\begin{remark}\label{fib}
The class $I\text{-}\mbox{inj}$, i.e. the class of trivial fibrations on $\ccat$ consist of the $\cc$-functors $F:\ca \rightarrow \cb$ such that:
\begin{itemize}
\item[-] for all objects $x,y \in \ca$, the morphism 
$$ F(x,y): \ca(x,y) \rightarrow \cb(Fx,Fy)$$
is a trivial fibration in $\cc$ and 
\item[-] the $\cc$-functor $F$ induces a surjective map on objects.
\end{itemize}

Notice also that the class $J\text{-}\mbox{inj}$, i.e. the class of fibrations on $\ccat$ consist of the $\cc$-functors $F:\ca \rightarrow \cb$ such that:
\begin{itemize} 
\item[-] for all objects $x,y \in \ca$, the morphism 
$$ F(x,y): \ca(x,y) \rightarrow \cb(Fx,Fy)$$
is a fibration in $\cc$ and 
\item[-] have the right lifting property (=R.L.P.) with respect to the set $J''$.
\end{itemize}
In particular, by condition $H2)$, a $\cc$-category $\ca$ is fibrant if and only if for all objects $x,y \in \ca$, the object $\ca(x,y)$ is fibrant in $\cc$.
\end{remark}

\begin{example}\label{exampleD}
\begin{itemize}
\item[1)] If in the first example of~\ref{exampleC}, we consider for $\ccat$ the Quillen model structure of theorem~\cite[1.8]{These}, where the set $J''$ consists of the dg functor $F:\ca \rightarrow \mathcal{K}$ (see section $1.3$ of \cite{These}), then conditions $H1)\text{-}H3)$ are satisfied: condition $H1)$ is satisfied by construction and conditions $H2)$ and $H3)$ follow from remark~\cite[1.14]{These} and corollary~\cite[13.1.3]{Hirschhorn}.

\item[2)] If in the second example of~\ref{exampleC}, we consider for $\ccat$ the Quillen model structure of theorem~\cite[4.7]{DGScat}, where the set $J''$ consists of the dg functor $F :\ca \rightarrow \mathcal{K}$, then conditions $H1)\text{-}H3)$ are satisfied: condition $H1)$ is also satisfied by construction (see section $4$ of \cite{DGScat}) and conditions $H2)$ and $H3)$ are analogous to the previous example.

\item[3)] If in the third example of~\ref{exampleC}, we consider for $\ccat$ the Quillen model structure of theorem~\cite[5.10]{DGScat}, where the set $J''$ consists of the simplicial abelian functor $L(F):L(\ca) \rightarrow L(\mathcal{K})$ (see section $5.2$ of \cite{DGScat}), then conditions $H1)\text{-}H3)$ are satisfied: since the normalization functor (see section $5.2$ of~\cite{DGScat})
$$N:s\mathbf{Ab}\text{-}\mbox{Cat} \longrightarrow \mathsf{dgcat}_{\geq 0}$$
preserves filtered colimits, admits a left Quillen adjoint and every object is fibrant in $s\mathbf{Ab}\text{-}\mbox{Cat}$ (see remark \cite[5.18]{DGScat}) conditions $H1)$-$H3)$ are satisfied.
\end{itemize}
\end{example}
From now on until the end of this chapter, we suppose that $\cc$, $K$ and $\ccat$ satisfy conditions $C1)\text{-}C7)$ and $H1)\text{-}H3)$.
\subsection{Levelwise quasi-equivalences}
In this section we consider the level model structure on $\spc$, see \cite[8.3]{Hovey}. Recall from \cite[7.2]{Hovey} that the category $\spc$ of $K$-symmetric spectra over $\cc$ is endowed with a symmetric monoidal structure $-\wedge-$ whose identity is the $K$-spectrum
$$ {\bf 1}_{\spc}=({\bf 1}_{\cc}, K, K\otimes K, \ldots, K^{\otimes^n}, \ldots )\,,$$
where the permutation group $\Sigma_n$ acts on $K^{\otimes^n}$ by permutation.
Moreover, as it is shown in \cite[7.3]{Hovey}, we have an adjunction
$$
\xymatrix{
\spc \ar@<1ex>[d]^{Ev_0} \\
\cc \ar@<1ex>[u]^{F_0} \,,
}
$$
where both adjoints are strong monoidal (\ref{lax}). This naturally induces, by remark~\ref{corleftadj}, the following adjunction
$$
\xymatrix{
\spcc \ar@<1ex>[d]^{Ev_0} \\
\cc\text{-}\mbox{Cat} \ar@<1ex>[u]^{F_0} \,.
}
$$
\begin{remark}\label{rk2}
By remark~\ref{rk1}, the right adjoint functor $Ev_0$ preserves filtered colimits.
\end{remark}
We now construct a Quillen model structure on $\spcc$, whose weak equivalence are defined as follows.
\begin{definition}\label{leveleq}
A $\spc$-functor $F:\ca \rightarrow \cb$ is a {\em levelwise quasi-equivalence} if:
\begin{itemize}
\item[L1)] for all objects $x, y \in \ca$, the morphism
$$ F(x,y):\ca(x,y) \rightarrow \cb(Fx,Fy)$$
is a level equivalence in $\spc$ (see \cite[8.1]{Hovey}) and
\item[L2)] the induced $\cc$-functor
$$ Ev_0(F):Ev_0(\ca) \rightarrow Ev_0(\cb)$$
is a weak equivalence (\ref{we}) in $\ccat$.
\end{itemize}
\end{definition}
\begin{notation}
We denote by $\cw_L$ the class of levelwise quasi-equivalences. 
\end{notation}
\begin{remark}
If condition $L1)$ is verified, condition $L2)$ is equivalent to:
\begin{itemize}
\item[L2')] the induced functor (\ref{[]})
$$ [Ev_0(F)]: [Ev_0(\ca)] \longrightarrow [Ev_0(\cb)]$$
is essentially surjective.
\end{itemize}
\end{remark}
We now define our sets of generating (trivial) cofibrations. 
\begin{definition}\label{defc}
The set $I_L$ of {\em generating cofibrations} in $\spcc$ consists of:
\begin{itemize}
\item[-] the set $I'_L$ of $\spc$-functors obtained by applying the functor $U$ (see \ref{funcU}) to the set $I_{\spc}$ of generating cofibrations for the level model structure on $\spc$ (see \cite[8.2]{Hovey}) and 
\item[-] the $\spc$-functor
$$\emptyset \longrightarrow \underline{{\bf 1}_{\spc}}$$
from the empty $\spc$-category $\emptyset$ (which is the initial object in $\spcc$) to the $\spc$-category $\underline{{\bf 1}_{\spc}}$ with one object $\ast$ and endomorphism ring ${\bf 1}_{\spc}$.
\end{itemize}
\end{definition}

\begin{definition}\label{defa}
The set $J_L:=J'_L\cup J''_L$ of {\em generating trivial cofibrations} in $\spcc$ consists of:
\begin{itemize}
\item[-] the set $J'_L$ of $\spc$-functors obtained by applying the functor $U$ to the set $J_{\spc}$ of trivial generating cofibrations for the level model structure on $\spc$ and 
\item[-] the set $J''_L$ of $\spc$-functors obtained by applying the functor $F_0$ to the set $J''$ of trivial generating cofibrations (\ref{gencofCcat}) in $\ccat$. 
\end{itemize}
\end{definition}
\begin{theorem}\label{levelwise}
If we let $\cm$ be the category $\spcc$, $W$ be the class $\cw_L$, $I$ be the set $I_L$ of definition~\ref{defc} and $J$ the set $J_L$ of definition~\ref{defa}, then the conditions of the recognition theorem \cite[2.1.19]{Hovey} are satisfied. Thus, the category $\spcc$ admits a cofibrantly generated Quillen model structure whose weak equivalences are the levelwise quasi-equivalences.
\end{theorem}
\subsection{Proof of Theorem \ref{levelwise}}
Observe that the category $\spcc$ is complete and cocomplete and that the class $\cw_L$ satisfies the two out of three axiom and is stable under retracts. Since the domains of the generating (trivial) cofibrations in $\spc$ are sequentially small (see~\ref{rk1}) the same holds by remark~\ref{rkA}, for the domains of $\spc$-functors in the sets $I'_L$ and $J'_L$. By remark~\ref{rk2}, this also holds for the remaining $\spc$-functors of the sets $I_L$ and $J_L$. This implies that the first three conditions of the recognition theorem \cite[2.1.19]{Hovey} are verified.

We now prove that $J_L\text{-}\mbox{inj} \cap \cw_L = I_L\text{-}\mbox{inj}$. For this we introduce the following auxiliary class of $\spc$-functors:
\begin{definition}
Let $\mathbf{Surj}$ be the class of $\spc$-functors $F:\ca \rightarrow \cb$ such that:
\begin{itemize}
\item[Sj1)] for all objects $x, y \in \ca$, the morphism 
$$F(x,y): \ca(x,y) \rightarrow \cb(Fx,Fy)$$
is a trivial fibration in $\spc$ and
\item[Sj2)] the $\spc$-functor $F$ induces a surjective map on objects.
\end{itemize}
\end{definition}

\begin{lemma}\label{Iinj}
$I_L\text{-}\mbox{inj} = \mathbf{Surj}\,.$
\end{lemma}
\begin{proof}
Notice that by adjunction (see remark~\ref{rkA}), a $\spc$-functor satisfies condition $Sj1)$ if and only if it has the R.L.P. with respect to the set $I'_L$ of generating cofibrations. Clearly a $\spc$-functor has the R.L.P. with respect to $\emptyset \rightarrow \underline{{\bf 1}_{\spc}}$ if and only if it satisfies condition $Sj2)$.
\end{proof}

\begin{lemma}\label{total}
$\mathbf{Surj}=J_L\text{-}\mbox{inj}\cap \cw_L\,.$
\end{lemma}

\begin{proof}
We prove first the inclusion $\subseteq$. Let $F:\ca \rightarrow \cb$ be a $\spc$-functor which belongs to $\mathbf{Surj}$. Clearly condition $Sj1)$ and $Sj2)$ imply conditions $L1)$ and $L2)$ and so $F$ belongs to $\cw_L$. Notice also that a $\spc$-functor which satisfies condition $Sj1)$ has the R.L.P. with respect to the set $J'_L$ of generating trivial cofibrations. It is then enough to show that $F$ has the R.L.P. with respect to the set $J''_L$ of generating trivial cofibrations.

By adjunction, this is equivalent to demand that the $\cc$-functor $Ev_0(F): Ev_0(\ca) \rightarrow Ev_0(\cb)$ has the R.L.P. with respect to the set $J''$ of generating trivial cofibrations in $\cc\text{-}\mbox{Cat}$. Since $F$ satisfies conditions $Sj1)$ and $Sj2)$, remark \ref{fib} implies that $Ev_0(F)$ is a trivial fibration in $\cc\text{-}\mbox{Cat}$ and so the claim follows.

We now prove the inclusion $\supseteq$. Start by observing that a $\spc$-functor  satisfies condition $Sj1)$ if and only if it satisfies condition $L1)$ and it has the R.L.P. with respect to the set $J'_L$ of generating trivial cofibrations. Now, let $F:\ca \rightarrow \cb$ be a $\spc$-functor which belongs to $J_L\text{-}\mbox{inj}\cap \cw_L$. It is then enough to show that it satisfies condition $Sj2)$.  Since $F$ has the R.L.P. with respect to the set $J_L$ of generating trivial cofibrations, the $\cc$-functor $Ev_0(F):Ev_0(\ca) \rightarrow Ev_0(\cb)$ has the R.L.P. with respect to the set $J$ of generating trivial cofibrations in $\cc\text{-}\mbox{Cat}$. These remarks imply that $Ev_0(F)$ is a trivial fibration in $\cc\text{-}\mbox{Cat}$ and so by remark \ref{fib}, the $\cc$-functor $Ev_0(F)$ induces a surjective map on objects. Since $Ev_0(F)$ and $F$ induce the same map on the set of objects, the $\spc$-functor $F$ satisfies condition $Sj2)$.
\end{proof}
We now characterize the class $J_L\text{-}\mbox{inj}$.
\begin{lemma}\label{fibrations}
A $\spc$-functor $F:\ca \rightarrow \cb$ has the R.L.P. with respect to the set $J_L$ of trivial generating cofibrations if and only if it satisfies:
\begin{itemize}
\item[F1)] for all objects $x,y \in \ca$, the morphism 
$$ F(x,y): \ca(x,y) \rightarrow \cb(Fx,Fy)$$
is a fibration in the level model structure on $\spc$ and
\item[F2)] the induced $\cc$-functor 
$$ Ev_0(F): Ev_0(\ca) \rightarrow Ev_0(\cb)$$
is a fibration (\ref{fib}) in the Quillen model structure on $\cc\text{-}\mbox{Cat}$.
\end{itemize}
\end{lemma}
\begin{proof}
Observe that a $\spc$-functor $F$ satisfies condition $F1)$ if and only if it has the R.L.P. with respect to the set $J'_L$ of generating trivial cofibrations. By adjunction, $F$ has the R.L.P. with respect to the set $J''_L$ if and only if the $\cc$-functor $Ev_0(F)$ has the R.L.P. with respect to the set $J''$. In conclusion $F$ has the R.L.P. with respect to the set $J_L$ if and only if it satisfies conditions $F1)$ and $F2)$ altogether.
\end{proof}
\begin{lemma}\label{cell}
$J'_L\text{-}\mbox{cell} \subseteq \cw_L\,.$
\end{lemma}
\begin{proof}
Since the class $\cw_L$ is stable under transfinite compositions (\cite[10.2.2]{Hirschhorn}) it is enough to prove the following: consider the following pushout:
$$ 
\xymatrix{
U(X) \ar[d]_{U(j)} \ar[r]^-R \ar@{}[dr]|{\lrcorner} & \ca \ar[d]^P \\
U(Y) \ar[r] & \cb\,,
}
$$
where $j$ belongs to the set $J_{\spc}$ of generating trivial cofibrations on $\spc$. We need to show that $P$ belongs to $\cw_L$. Since the morphism $j: X \longrightarrow Y$ is a trivial cofibration in $\spc$, proposition~\ref{clef} and proposition~\cite[8.3]{Hovey} imply that the spectral functor $P$ satisfies condition $L1)$. Since $P$ induces the identity map on objects, condition $L2')$ is automatically satisfied and so $P$ belongs to $\cw_L$.
\end{proof}

\begin{proposition}\label{cell1}
$J''_L\text{-}\mbox{cell} \subseteq \cw_L\,.$
\end{proposition}
\begin{proof}
Since the class $\cw_L$ is stable under transfinite compositions, it is enough to prove the following: consider the following push-out
$$
\xymatrix{
F_0(\cg) \ar[d]_{F_0(V)} \ar[r]^R \ar@{}[dr]|{\lrcorner} & \ca \ar[d]^P\\
F_0(\cd) \ar[r] & \cb \,,
}
$$
where $\cg \stackrel{V}{\rightarrow} \cd$ belongs to the set $J''$ of generating trivial cofibrations in $\cc\text{-}\mbox{Cat}$. 
We need to show that $P$ belongs to $\cw_L$. We start by condition $L1)$. Factor the $\mathsf{Sp}^{\Sigma}(\cc,K)$-functor $F_0(V)$ as 
$$ F_0(\cg) \rightarrow F_0(\cd)(1) \hookrightarrow F_0(\cd)\,,$$
where $F_0(\cd)(1)$ is the full $\mathsf{Sp}^{\Sigma}(\cc, K)$-subcategory of $F_0(\cd)$ whose set of objects is $\{1\}$, see condition $H1)$. Consider the iterated pushout:
$$
\xymatrix{
*+<1pc>{F_0(\cg)} \ar@{>->}[d]^{\sim} \ar[r]^R \ar@{}[dr]|{\lrcorner} & \ca \ar[d]_{P_0} \ar@/^1pc/[dd]^P\\
F_0(\cd)(1) \ar[r] \ar@{^{(}->}[d] \ar@{}[dr]|{\lrcorner} & \widetilde{\ca} \ar[d]_{P_1} \\
F_0(\cd) \ar[r] & \cb \,.
}
$$
In the lower pushout, since $F_0(\cd)(1)$ is a full $\mathsf{Sp}^{\Sigma}(\cc, K)$-subcategory of $F_0(\cd)$, proposition~\cite[5.2]{Latch} implies that $\widetilde{\ca}$ is a full $\mathsf{Sp}^{\Sigma}(\cc,K)$-subcategory of $\cb$ and so $P_1$ satisfies condition $L1)$. In the upper pushout, by condition $H1)$ the map $\xymatrix{*+<1pc>{\cg(1)} \ar@{>->}[r]^{\sim} & \cd(1)}$ is a trivial cofibration so it is $\xymatrix{*+<1pc>{F_0(\cg)} \ar@{>->}[r]^{\sim} & F_0(\cd)(1)}$.

Now, let $\co$ denote the set of objects of $\ca$ and let $\co':= \co \backslash R(1)$. Observe that $\widetilde{\ca}$ is identified with the following pushout in $\mathsf{Sp}^{\Sigma}(\cc,K)^{\co}\text{-}\mbox{Cat}$ (see remark~\ref{nova1})
$$
\xymatrix{
\underset{\co'}{\coprod} \,F_0(\cg) \amalg F_0(\cg) \ar@{>->}[d]^{\sim} \ar[rr]^-R \ar@{}[drr]|{\lrcorner} && \ca \ar[d]^{P_0}\\
\underset{\co'}{\coprod} \,F_0(\cg) \amalg F_0(\cd)(1) \ar[rr] && \widetilde{\ca}\,.
}
$$
Since the left vertical arrow is a trivial cofibration so it is $P_0$. In particular $P_0$ satisfies condition $L_1)$ and so the composed dg functor $P$ satisfies also condition $L_1)$.

We now show that $P$ satisfies condition $L2')$. Consider the following commutative square in $\mbox{Cat}$
$$
\xymatrix{
[\cg] \ar[rrr]^-{[Ev_0(R)]} \ar[d]_{[V]} &&& [Ev_0(\ca)] \ar[d]^{[Ev_0(P)]} \\
[\cd] \ar[rrr] &&&  [Ev_0(\cb)]\,.
}
$$
We need to show that the functor $[Ev_0(P)]$ is essentially surjective.
Notice that if we restrict ourselves to the objects of each category $(\mbox{obj}(-))$ in the previous diagram, we obtain the following co-cartesien square
$$
\xymatrix{
\mbox{obj}\,[\cg] \ar[d]_{\mbox{obj}\,[V]} \ar[rrr]^{\mbox{obj}\,[Ev_0(R)]} \ar@{}[drrr]|{\lrcorner} & & & \mbox{obj}\,[Ev_0(\ca)] \ar[d]^{\mbox{obj}\,[Ev_0(P)]} \\
\mbox{obj}\,[\cd] \ar[rrr] & & &  \mbox{obj}\,[Ev_0(\cb)]
}
$$
in $\mathbf{Set}$. Since $V$ belongs to $J''$, the functor $[V]$ is essentially surjective. These facts imply, by a simple diagram chasing argument, that $[Ev_0(P)]$ is also essentially surjective and so the $\spc$-functor $P$ satisfies condition $L2')$.

In conclusion, $P$ satisfies condition $L1)$ and $L2')$ and so it belongs to $\cw_L$.
\end{proof}

We have shown that $J_L\text{-}\mbox{cell} \subseteq \cw_L$ (lemma~\ref{cell} and proposition~\ref{cell1}) and that $I_L\text{-}\mbox{inj} = J_L\text{-}\mbox{inj} \cap \cw_L$ (lemmas~\ref{Iinj} and \ref{total}). This implies that the last three conditions of the recognition theorem \cite[2.1.19]{Hovey} are satisfied. This finishes the proof of theorem~\ref{levelwise}.

\subsection{Properties I}

\begin{proposition}\label{Fibrations}
A $\spc$-functor $F:\ca \rightarrow \cb$ is a fibration with respect to the model structure of theorem~\ref{levelwise}, if and only if it satisfies conditions $F1)$ and $F2)$ of lemma~\ref{fibrations}.
\end{proposition}
\begin{proof}
This follows from lemma~\ref{fibrations}, since by the recognition theorem~\cite[2.1.19]{Hovey}, the set $J_L$ is a set of generating trivial cofibrations.
\end{proof}

\begin{corollary}\label{fiblev}
A $\spc$-category $\ca$ is fibrant with respect to the model structure of theorem~\ref{levelwise}, if and only if for all objects $x,y \in \cc$, the $K$-spectrum $\ca(x,y)$ is levelwise fibrant in $\spc$ (see \cite[8.2]{Hovey}). 
\end{corollary}
\begin{remark}\label{Qadj1}
Notice that by proposition~\ref{Fibrations} and remark~\ref{fib} we have a Quillen adjunction
$$
\xymatrix{
\spcc \ar@<1ex>[d]^{Ev_0} \\
\cc \text{-}\mbox{Cat} \ar@<1ex>[u]^{F_0}\,. 
}
$$
\end{remark}

\begin{proposition}\label{Rproper}
The Quillen model structure on $\spcc$ of theorem~\ref{levelwise} is right proper.
\end{proposition}

\begin{proof}
Consider the following pullback square in $\spcc$
$$
\xymatrix{
\ca \underset{\cb}{\times}\cc \ar[d] \ar[r]^-P \ar@{}[dr]|{\ulcorner} & \cc \ar@{->>}[d]^F\\
\ca \ar[r]_R^{\sim} & \cb
}
$$
with $R$ a levelwise quasi-equivalence and $F$ a fibration. We now show that $P$ is a levelwise quasi-equivalence. Notice that pullbacks in $\spcc$ are calculated on objects and on morphisms. Since the level model structure on $\spc$ is right proper (see~\ref{rk1}) and $F$ satisfies condition $F1)$, the spectral functor $P$ satisfies condition $L1)$.

Notice that if we apply the functor $Ev_0$ to the previous diagram, we obtain a pullback square in $\cc\text{-}\mbox{Cat}$
$$
\xymatrix{
Ev_0(\ca) \underset{Ev_0(\cb)}{\times} Ev_0(\cc) \ar[d] \ar[r]^-{Ev_0(P)} \ar@{}[dr]|{\ulcorner} & \cc \ar@{->>}[d]^{Ev_0(F)}\\
Ev_0(\ca) \ar[r]_{Ev_0(R)} & Ev_0(\cb)
}
$$
with $Ev_0(R)$ a weak equivalence and $Ev_0(F)$ a fibration. Now, condition $C3)$ allow us to conclude that $P$ satisfies condition $L2)$.
\end{proof}

\begin{proposition}\label{cof}
Let $\ca$ be a cofibrant $\spc$-category (in the Quillen model structure of theorem~\ref{levelwise}). Then for all objects $x,y \in \ca$, the $K$-spectrum $\ca(x,y)$ is cofibrant in the level model structure on $\spc$.
\end{proposition}
\begin{proof}
The model structure of theorem~\ref{levelwise} is cofibrantly generated and so any cofibrant object in $\spcc$ is a retract of a $I_L$-cell complex, see corollary \cite[11.2.2]{Hirschhorn}. Since cofibrations are stable under filtered colimits it is enough to prove the proposition for pushouts along a generating cofibration. Let $\ca$ be a $\spc$-category such that $\ca(x,y)$ is cofibrant for all objects $x,y \in \ca$:
\begin{itemize}
\item[-] Consider the following pushout
$$ 
\xymatrix{
\emptyset \ar[d] \ar[r] \ar@{}[dr]|{\lrcorner} & \ca \ar[d] \\
\underline{{\bf 1}_{\spc}} \ar[r] & \cb\,.
}
$$
Notice that $\cb$ is obtained from $\ca$, by simply introducing a new object. It is then clear that, for all objects $x,y \in \cb$, the $K$-spectrum $\cb(x,y)$ is cofibrant in the level model structure on $\spc$.
\item[-] Now, consider the following pushout
$$
\xymatrix{
U(X) \ar[r] \ar[d]_{i} \ar@{}[dr]|{\lrcorner} & \ca \ar[d]^P\\
U(Y) \ar[r] & \cb\,,
}
$$
where $i:X \rightarrow Y$ belongs to the set $I_{\spc}$ of generating cofibrations of $\spc$. Notice that $\ca$ and $\cb$ have the same set of objects and $P$ induces the identity map on the set of objects.
Since $i : X \rightarrow Y$ is a cofibration, proposition~\ref{clef1} and proposition~\cite[8.3]{Hovey} imply that the morphism 
$$P(x,y):\ca(x,y) \longrightarrow \cb(x,y)$$
is still a cofibration. Since $I$-cell complexes in $\spcc$ are built of $\emptyset$ (the initial object), the proposition is proven.
\end{itemize}
\end{proof}

\begin{lemma}\label{U1}
The functor
$$ U: \spc \longrightarrow \spcc\,\,\,\,\,\,(\mbox{see} \,\,\ref{funcU})$$
sends cofibrations to cofibrations.
\end{lemma}
\begin{proof}
The model structure of theorem~\ref{levelwise} is cofibrantly generated and so any cofibrant object in $\spcc$ is a rectract of a (possibly infinite) composition of pushouts along the generating cofibrations. Since the functor $U$ preserves retractions, colimits and send the generating cofibrations to (generating) cofibrations (see \ref{defc}) the lemma is proven.
\end{proof}
\subsection{Stable quasi-equivalences}
In this section we consider the stable model structure on $\spc$, see \cite[8.7]{Hovey}. Recall from \cite[7.2]{Hovey} that since $\cc$ is a closed symmetric monoidal category, the category $\mathsf{Sp}^{\Sigma}(\cc,K)$ is tensored, co-tensored and enriched over $\cc$. In what follows, we denote by $\underline{\mbox{Hom}}(-,-)$ this $\cc$-enrichment and by $\underline{\mbox{Hom}}_{\cc}(-,-)$ the internal $\mbox{Hom}$-object in $\cc$.

We will now construct a Quillen model structure on $\spcc$, whose weak equivalences are as follows.
\begin{definition}\label{we2}
A $\spc$-functor $F$ is a {\em stable quasi-equivalence} if it is a weak equivalence in the sense of definition~\ref{we}.
\end{definition}
\begin{notation}
We denote by $\cw_S$ the class of stable quasi-equivalences.
\end{notation}
\subsection{$Q$-functor}
We now construct a functor
$$ Q:\spcc \longrightarrow \spcc$$
and a natural transformation $\eta:\mbox{Id} \rightarrow Q$, from the identity functor on $\spcc$ to $Q$. 
Recall from \cite[8.7]{Hovey} the following set of morphisms in $\spc$:
$$ \cs:=\{ F_{n+1}(C\otimes K) \stackrel{\zeta_n^C}{\longrightarrow} F_nC\,|\,n \geq 0 \}\,,$$
where $C$ runs through the domains and codomains of the generating cofibrations of $\cc$ and $\zeta_n^C$ is the adjoint to the map 
$$C \otimes K \longrightarrow Ev_{n+1}F_n(C)=\Sigma_{n+1}\times(C\otimes K)$$
corresponding to the identity of $\Sigma_{n+1}$.
For each element of $\cs$, we consider a factorization
$$
\xymatrix{
F_{n+1}(C\otimes K) \ar[rr]^{\zeta^C_n} \ar@{>->}[dr]_{\widetilde{\zeta_n^C}} && F_n(C) \\
& Z_n^C \ar@{->>}[ur]_{\sim} &
}
$$
in the level model structure on $\spc$. Recall from \cite[8.7]{Hovey}, that the stable model structure on $\spc$ is the left Bousfield localization of the level model structure with respect to the morphisms $\zeta^C_n$ (or with respect to the morphisms $\widetilde{\zeta^C_n}$).
\begin{definition}
Let $\widetilde{\cs}$ be the following set of morphisms in $\mathsf{Sp}^{\Sigma}(\cc,K)$:
$$
i \otimes \widetilde{\zeta^C_n}: B \otimes F_{n+1}(C\otimes K) \underset{A\otimes F_{n+1}(C\otimes K)}{\coprod}A\otimes Z_n^C \longrightarrow  B\otimes Z_n^C\,,
$$
where $i:A \rightarrow B$ is a generating cofibration in $\cc$.
\end{definition}

\begin{remark}
Since $\mathsf{Sp}^{\Sigma}(\cc,K)$ is a monoidal model category, the elements of the set $\widetilde{\cs}$ are cofibrations with cofibrant (and sequentially small) domains.
\end{remark}

\begin{definition}\label{OmegaF}
Let $\ca$ be a $\spc$-category. The functor $Q:\spcc \rightarrow \spcc$ is obtained by applying the small object argument, using the set $J'_L \cup U(\widetilde{\cs})$ (see \ref{funcU}) to factorize the $\spc$-functor
$$ \ca \longrightarrow \bullet,$$
where $\bullet$ denotes the terminal object. 
\end{definition}
\begin{remark}\label{Omeg}
We obtain in this way a functor $Q$ and a natural transformation $\eta: \mbox{Id} \rightarrow Q$. Notice also that $Q(\ca)$ has the same objects as $\ca$, and the R.L.P. with respect to the set $J'_L \cup U(\widetilde{\cs})$.
\end{remark}

\begin{proposition}\label{omega}
Let $\ca$ be a $\spc$-category which has the right lifting property with respect to the set $J'_L \cup U(\widetilde{\cs})$. Then it satisfies the following condition:
\begin{itemize}
\item[$\Omega)$] for all objects $x,y \in \ca$, the $K$-spectrum $\ca(x,y)$ is an $\Omega$-spectrum (see \cite[8.6]{Hovey}).
\end{itemize}
\end{proposition}
\begin{proof}
By corollary~\ref{fiblev}, $\ca$ has the R.L.P. with respect to $J'_L$ if and only if for all objects $x,y \in \ca$, the $K$-spectrum $\ca(x,y)$ is levelwise fibrant in $\spc$.

Now, suppose that $\ca$ has the R.L.P. with respect to $U(\widetilde{\cs})$. Then by adjunction, for all objects $x,y \in \ca$, the induced morphism
$$ \underline{\mbox{Hom}}(Z^C_n, \ca(x,y)) \longrightarrow \underline{\mbox{Hom}}(F_{n+1}(C\otimes K), \ca(x,y))$$
is a trivial fibration in $\cc$. Notice also that we have the following weak equivalences
$$
\begin{array}{rcl}
\underline{\mbox{Hom}}(Z^C_n, \ca(x,y)) & \simeq & \underline{\mbox{Hom}}(F_n(C), \ca(x,y)) \\
 & \simeq & \underline{\mbox{Hom}}_{\cc}(C, \ca(x,y)_n)
\end{array}
$$
and
$$
\begin{array}{rcl}
\underline{\mbox{Hom}}(F_{n+1}(C\otimes K), \ca(x,y)) & \simeq & \underline{\mbox{Hom}}_{\cc}(C \otimes K, \ca(x,y)_{n+1}) \\
 & \simeq & \underline{\mbox{Hom}}_{\cc}(C, \ca(x,y)^K_{n+1})\,.
\end{array}
$$
This implies that the induced morphisms in $\cc$
$$ \underline{\mbox{Hom}}_{\cc}(C,\ca(x,y)_n) \stackrel{\sim}{\longrightarrow} \underline{\mbox{Hom}}_{\cc}(C,\ca(x,y)_{n+1}^K)\,,$$
are weak equivalences, where $C$ runs through the domains and codomains of the generating cofibrations of $\cc$.
By condition $C7)$, we obtain the weak equivalence
$$\ca(x,y)_n \simeq \underline{\mbox{Hom}}_{\cc}({\bf 1}_{\cc},\ca(x,y)_n) \stackrel{\sim}{\longrightarrow} \underline{\mbox{Hom}}_{\cc}({\bf 1}_{\cc},\ca(x,y)_{n+1}^K) \simeq \ca(x,y)^K_{n+1}\,.$$
In conclusion, for all objects $x,y \in \ca$, the $K$-spectrum $\ca(x,y)$ is an $\Omega$-spectrum.
\end{proof}

\begin{proposition}\label{fibres}
Let $\ca$ be a $\spc$-category. The $\spc$-functor
$$ \eta_{\ca}: \ca \longrightarrow Q(\ca)$$
is a stable quasi-equivalence (\ref{we2}).
\end{proposition}
\begin{proof}
Notice that the elements of the set $J_{\spc}\cup \widetilde{\cs} $ are trivial cofibrations in the stable model structure on $\spc$ (see \cite[8.8]{Hovey}). Notice also that the stable model structure is monoidal (\cite[8.11]{Hovey}) and that by condition $C6)$, it satisfies the monoid axiom. This implies, by proposition~\ref{clef}, that $\eta_{\ca}$ satisfies condition $WE1)$. Since the $\spc$-categories $Q(\ca)$ and $\ca$ have the same set of objects condition $WE2')$ is automatically verified.
\end{proof}
\subsection{Main theorem}
\begin{definition}
A $\spc$-functor $F:\ca \rightarrow \cb$ is:
\begin{itemize}
\item[-] a {\em $Q$-weak equivalence} if $Q(F)$ is a levelwise quasi-equivalence (\ref{leveleq}).
\item[-] a {\em cofibration} if it is a cofibration in the model structure of theorem~\ref{levelwise}.
\item[-] a {\em $Q$-fibration} if it has the R.L.P. with respect to all cofibrations which are $Q$-weak equivalences.
\end{itemize}
\end{definition}
\begin{lemma}
Let $\ca$ be a $\spc$-category which satisfies condition $\Omega)$ of proposition~\ref{omega}. Then the category $[\ca]$ (see \ref{[]}) is naturally identified with the category $[Ev_0(\ca)]$.
\end{lemma}
\begin{proof}
Recall from \cite[8.8]{Hovey}, that an object in $\spc$ is stably fibrant if and only if it is an $\Omega$-spectrum. Since the adjunction
$$
\xymatrix{
\mathsf{Sp}^{\Sigma}(\cc,K) \ar@<1ex>[d]^{Ev_0} \\
\cc \ar@<1ex>[u]^{F_0}
}
$$
is strong monoidal, we have the following identifications
$$ [ {\bf 1}_{\spc}, \ca(x,y)] \simeq [{\bf 1}_{\cc}, Ev_0(\ca(x,y))] \,,\,\,\, x,y \in \ca\,,$$
which imply the lemma.
\end{proof}
\begin{corollary}\label{[]2}
Let $F:\ca \rightarrow \cb$ be a $\spc$-functor between $\spc$-categories which satisfy condition $\Omega)$. Then $F$ satisfies condition $WE2)$ if and only if $Ev_0(F)$ satisfies condition $WE2)$.
\end{corollary}

\begin{proposition}\label{weak=level}
Let $F:\ca \rightarrow \cb$ be stable quasi-equivalence between $\spc$-categories which satisfy the condition $\Omega)$. Then $F$ is a levelwise quasi-equivalence (\ref{leveleq}).
\end{proposition}
\begin{proof}
Since $F$ satisfies condition $WE1)$ and $\ca$ and $\cb$ satisfy condition $\Omega)$, lemmas \cite[4.3.5-4.3.6]{Hirschhorn} imply that $F$ satisfies condition $L1)$. By corollary~\ref{[]2}, the $\spc$-functor $F$ satisfies condition $WE2)$ if and only if it satisfies condition $L2)$. This proves the proposition. 
\end{proof}

\begin{lemma}\label{coincide1}
A $\spc$-functor $F:\ca \rightarrow \cb$ is a $Q$-weak equivalence if and only if it is a stable quasi-equivalence.
\end{lemma}
\begin{proof}
We have at our disposal a commutative square
$$
\xymatrix{
\ca \ar[d]_F \ar[r]^-{\eta_{\ca}} & Q(\ca) \ar[d]^{Q(F)} \\
\cb \ar[r]_-{\eta_{\cb}} & Q(\cb) 
}
$$
where the $\spc$-functors $\eta_{\ca}$ and $\eta_{\cb}$ are stable quasi-equivalences by proposition~\ref{fibres}. Since the class $\cw_S$ satisfies the two out of three axiom, the $\spc$-functor $F$ is a stable quasi-equivalence if and only if $Q(F)$ is a stable quasi-equivalence. The $\spc$-categories $Q(\ca)$ and $Q(\cb)$ satisfy condition $\Omega)$ and so by proposition~\ref{weak=level}, $Q(F)$ is a levelwise quasi-equivalence. This proves the lemma.
\end{proof}
\begin{theorem}\label{stable}
The category $\spcc$ admits a right proper Quillen model structure whose weak equivalences are the stable quasi-equivalences (\ref{we2}) and the cofibrations those of theorem~\ref{levelwise}.
\end{theorem}
\begin{proof}
The proof will consist on verifying the conditions of theorem~\ref{modif}. We consider for $\cm$ the Quillen model structure of theorem~\ref{levelwise} and for $Q$ and $\eta$, the functor and natural transformation defined in \ref{OmegaF}. The Quillen model structure of theorem~\ref{levelwise} is right proper (see~\ref{Rproper}) and by lemma~\ref{coincide1} the $Q$-weak equivalences are precisely the stable quasi-equivalences. We now verify conditions (A1), (A2) and (A3):
\begin{itemize}
\item[(A1)] Let $F:\ca \rightarrow \cb$ be a levelwise quasi-equivalence. We have the following commutative square
$$
\xymatrix{
\ca \ar[d]_F \ar[r]^-{\eta_{\ca}} & Q(\ca) \ar[d]^-{Q(F)}\\
\cb \ar[r]_-{\eta_{\cb}} & Q(\cb)
}
$$
with $\eta_{\ca}$ and $\eta_{\cb}$ stable quasi-equivalences. Notice that since $F$ is a levelwise quasi-equivalence it is in particular a stable quasi-equivalence. The class $\cw_S$ satisfies the three out of three axiom and so $Q(F)$ is a stable quasi-equivalence. Since the $Q(\ca)$ and $Q(\cb)$ satisfy condition $\Omega)$, proposition~\ref{weak=level} implies that $Q(F)$ is in fact a levelwise quasi-equivalence.

\item[(A2)] We now show that for every $\spc$-category $\ca$, the $\spc$-functors
$$\eta_{Q(\ca)}, Q(\eta_{\ca}):Q(\ca) \rightarrow QQ(\ca)$$
are levelwise quasi-equivalences.
Since the $\spc$-functors $\eta_{Q(\ca)}$ and $Q(\eta_{\ca})$ are stable quasi-equivalences between $\spc$-categories which satisfy condition $\Omega)$, proposition~\ref{weak=level} implies that they are levelwise quasi-equivalences.

\item[(A3)] We start by observing that if $P: \cc \rightarrow \cd$ is a $Q$-fibration, then for all objects $x,y \in \cc$, the morphism
$$ P(x,y): \cc(x,y) \longrightarrow \cd(Px,Py)$$
is a $Q$-fibration in $\spc$. In fact, by proposition~\ref{U1}, the functor
$$U: \spc \longrightarrow \spcc$$
sends cofibrations to cofibrations. Since clearly it sends stable equivalences in $\spc$ to stable quasi-equivalences in $\spcc$, the claim follows.

Now consider the diagram
$$
\xymatrix{
\ca \underset{Q(\ca)}{\times}\cb \ar[d] \ar[r] \ar@{}[dr]|{\ulcorner} & \cb \ar[d]^P\\
\ca \ar[r]_{\eta_{\ca}} & Q(\ca),
}
$$
with $P$ a $Q$-fibration. The stable model structure on $\spc$ is right proper by condition $C6)$, and so we conclude, by construction of fiber products in $\spcc$, that the induced $\spc$-functor
$$ \eta_{\ca_*}: \ca \underset{Q(\ca)}{\times}\cb \longrightarrow \cb$$
satisfies condition $WE1)$. Since $\eta_{\ca}$ induces the identity map on objects so thus $\eta_{\ca_*}$, and so condition $WE2')$ is proven.
\end{itemize}
\end{proof}
\subsection{Properties II}
\begin{proposition}\label{SfibS}
A $\spc$-category $\ca$ is fibrant with respect to theorem~\ref{stable} if and only if for all objects $x,y \in \ca$, the $K$-spectrum $\ca(x,y)$ is an $\Omega$-spectrum in $\spc$.
\end{proposition}
\begin{notation}
We denote these fibrant $\spc$-categories by {\em $Q$-fibrant}.
\end{notation}
\begin{proof}
By corollary~\ref{cormodif}, $\ca$ is fibrant with respect to theorem~\ref{stable} if and only if it is fibrant for the structure of theorem~\ref{levelwise} (see corollary~\ref{fiblev}) and the $\spc$-functor $\eta_{\ca}:\ca \rightarrow Q(\ca)$ is a levelwise quasi-equivalence. Observe that $\eta_{\ca}$ is a levelwise quasi-equivalence if and only if for all objects $x,y \in \ca$ the morphism
$$ \eta_{\ca}(x,y):\ca(x,y) \longrightarrow Q(\ca)(x,y)$$
is a level equivalence in $\spc$. Since $Q(\ca)(x,y)$ is an $\Omega$-spectrum we have the following commutative diagrams (for all $n\geq 0$)
$$
\xymatrix{
\ca(x,y)_n \ar[d] \ar[r]^-{\widetilde{\delta_n}} & \ca(x,y)_{n+1}^K \ar[d] \\
Q(\ca)(x,y)_n \ar[r]_-{\widetilde{\delta_n}} &  Q(\ca)(x,y)_{n+1}^K\,,
}
$$
where the bottom and vertical arrows are weak equivalences in $\cc$. This implies that
$$ \widetilde{\delta}_n: \ca(x,y)_n \longrightarrow \ca(x,y)_{n+1}^K, \,\,\, n \geq 0$$
is a weak equivalence in $\cc$ and so, for all $x,y \in \ca$, the $K$-spectrum $\ca(x,y)$ is an $\Omega$-spectrum in $\spc$.
\end{proof}
\begin{remark}\label{resolS}
Remark~\ref{Omeg} and propositions~\ref{omega} and \ref{fibres} imply that $\eta_{\ca}:\ca \rightarrow Q(\ca)$ is a (functorial) fibrant replacement of $\ca$ in the model structure of theorem~\ref{stable}. In particular $\eta_{\ca}$ induces the identity map on the set of objects.

Notice also, that since the cofibrations (and so the trivial fibrations) of the model structure of theorem~\ref{stable} are the same as those of theorem~\ref{levelwise}, proposition~\ref{cof}  and lemma~\ref{U1} stay valid in the Quillen model structure of theorem~\ref{stable}.
\end{remark}

\begin{proposition}\label{Qfibfib}
Let $F:\ca \rightarrow \cb$ is a $\spc$-functor between $Q$-fibrant $\spc$-categories. Then $F$ is a $Q$-fibration if and only if it is a fibration.
\end{proposition}
\begin{proof}
By theorem~\ref{Qfib}, if $F$ is a $Q$-fibration then it is a fibration. Let us now prove the converse. Suppose that $F:\ca \rightarrow \cb$ is a fibration. Consider the commutative square
$$
\xymatrix{
\ca \ar[d]_F \ar[r]^-{\eta_{\ca}} & Q(\ca) \ar[d]^{Q(F)} \\
\cb \ar[r]_-{\eta_{\cb}} & Q(\cb)\,.
}
$$
Since $\ca$ and $\cb$ satisfy the condition $\Omega)$, proposition~\ref{weak=level} implies that $\eta_{\ca}$ and $\eta_{\cb}$ are levelwise quasi-equivalences. Since the model structure of theorem~\ref{levelwise} is right proper (see~\ref{Rproper}), the previous square is homotopy cartesian and so by theorem~\ref{Qfib}, $F$ is in fact a $Q$-fibration.
\end{proof}
\begin{remark}\label{Qadj2}
Since a $Q$-fibration is a fibration, the adjunction of remark~\ref{Qadj1} induces a natural Quillen adjunction
$$
\xymatrix{
\spcc \ar@<1ex>[d]^{Ev_0} \\
\ccat \ar@<1ex>[u]^{F_0}
}
$$
with respect to the model structure of theorem~\ref{stable}.
\end{remark}
\subsection{Homotopy Idempotence}
In \cite[9.1]{Hovey}, Hovey proved the following idempotence property:
\begin{itemize}
\item[S)] If the functor 
$$-\otimes K: \cc \longrightarrow \cc$$
 is a Quillen equivalence.
\end{itemize}
Then the Quillen adjunction
$$
\xymatrix{
\spc \ar@<1ex>[d]^{Ev_0} \\
\cc \ar@<1ex>[u]^{F_0}\,,
}
$$
is a Quillen equivalence, where $\spc$ is endowed with the stable model structure.
\begin{theorem}\label{conditionS}
If conditon $S)$ is verified then the Quillen adjunction of remark~\ref{Qadj2}
$$
\xymatrix{
\spcc \ar@<1ex>[d]^{Ev_0} \\
\ccat \ar@<1ex>[u]^{F_0}
}
$$
is a Quillen equivalence. 
\end{theorem}
\begin{proof}
The proof will consist on verifying conditions $a)$ and $b)$ of proposition~\ref{Ncrit}.
\begin{itemize}
\item[a)] Let $F:\ca \rightarrow \cb$ be a $\spc$-functor between $Q$-fibrant $\spc$-categories, such that $Ev_0(\ca):Ev_0(\ca) \rightarrow Ev_0(\cb)$ is a weak equivalence in $\ccat$. Since $\ca$ and $\cb$ satisfy condition $\Omega)$, condition $a)$ of proposition~\ref{Ncrit} applied to the Quillen equivalence
$$ Ev_0:\mathsf{Sp}^{\Sigma}(\cc,K) \longrightarrow \cc$$
implies that $F$ satisfies condition $WE1)$. Notice also that since $Ev_0(F)$ is a weak equivalence in $\ccat$, corollary~\ref{[]2} implies that $F$ satisfies condition $WE2)$. In conclusion $F$ is a stable quasi-equivalence.
\item[b)] Let $\ca$ be a cofibrant $\cc$-category and let us denote by $I$ its set of objects. Notice, that in particular $\ca$ is cofibrant in $\cc^I\text{-}\mbox{Cat}$, with respect to the model structure of remark~\ref{nova1}. Now, since the adjunction
$$
\xymatrix{
\spc \ar@<1ex>[d]^{Ev_0} \\
\cc \ar@<1ex>[u]^{F_0}\,,
}
$$
is a strong monoidal Quillen equivalence (\ref{smwe}), remark~\ref{nova2} implies that condition $b)$ of proposition~\ref{Ncrit} is verified for the adjunction $(F_0^I,Ev_0^I)$. By remark~\ref{resolS}, the $\spc$-functor $\eta_{\ca}:\ca \rightarrow Q(\ca)$ is a fibrant resolution in $\spc^I$ and so we conclude that condition $b)$ is verified.
\end{itemize}
\end{proof}
\begin{remark}\label{glob1}
Clearly the first example of \ref{exampleC} satisfies condition $S)$, since the functor $-\otimes \mathbb{Z}[1]:Ch \rightarrow Ch$ is the classical suspension procedure. This implies, by theorem~\ref{conditionS}, that we have a Quillen equivalence
$$
\xymatrix{
\mathsf{Sp}^{\Sigma}(Ch,\mathbb[1])\text{-}\mbox{Cat} \ar@<1ex>[d]^{Ev_0} \\
\mathsf{dgcat} \ar@<1ex>[u]^{F_0} \,.
}
$$
\end{remark}
\section{A Quillen equivalence criterion}\label{chapter4}
Let $\cc$ and $\cd$ be two monoidal model categories with $K$ an object of $\cc$ and $K'$ an object of $\cd$.
In this chapter we establish a general criterion for a Quillen equivalence between $\mathsf{Sp}^{\Sigma}(\cc, K)$ and $\mathsf{Sp}^{\Sigma}(\cd, K')$ to induce a Quillen equivalence between $\mathsf{Sp}^{\Sigma}(\cc, K)\text{-}\mbox{Cat}$ and $\mathsf{Sp}^{\Sigma}(\cd,K')\text{-}\mbox{Cat}$ (see theorem~\ref{chave}).

Suppose that both $\cc$, $K$, $\ccat$ and $\cd$, $K'$, $\cd \text{-}\mbox{Cat}$ satisfy conditions $C1)\text{-}C7)$ and $H1)\text{-}H3)$ from the previous chapter. Suppose moreover, that we have a weak monoidal Quillen equivalence (\ref{smwe}) 
$$
\xymatrix{
\spc \ar@<1ex>[d]^N \\
\mathsf{Sp}^{\Sigma}(\cd, K') \ar@<1ex>[u]^L
}
$$
with respect to the stable model structures (see \cite[8.7]{Hovey}), which preserves fibrations and levelwise weak equivalences. By proposition~\ref{Adjunction}, we have the following adjunction
$$
\xymatrix{
\spcc \ar@<1ex>[d]^N \\
\mathsf{Sp}^{\Sigma}(\cd, K')\text{-}\mbox{Cat} \ar@<1ex>[u]^{L_{cat}}\,.
}
$$
\begin{theorem}\label{chave}
If there exists a functor 
$$\widetilde{N}: \ccat \longrightarrow \cd \text{-}\mbox{Cat}$$
which preserves fibrations and weak equivalences and makes the following diagram
$$
\xymatrix{
\ccat \ar[d]_{\widetilde{N}} && \spcc \ar[ll]_-{Ev_0} \ar[d]^N \\
\cd \text{-}\mbox{Cat} && \mathsf{Sp}^{\Sigma}(\cd,K')\text{-}\mbox{Cat} \ar[ll]^-{Ev_0}
}
$$
commute, then the adjunction $(L_{cat},N)$ is a Quillen adjunction (with respect to the model structure of theorem~\ref{stable}). 

Moreover, if $\widetilde{N}$ satisfies the following `reflexion' condition:
\begin{itemize}
\item[R)] If the $\cd$-functor $\widetilde{N}(F): \widetilde{N}(\ca) \rightarrow \widetilde{N}(\cb)$ satisfies condition $WE2)$, so does the $\cc$-functor $F:\ca \rightarrow \cb$.
\end{itemize}
Then the Quillen adjunction $(L_{cat},N)$ is a Quillen equivalence.
\end{theorem}
\begin{proof}
We start by proving that the adjunction $(L_{cat},N)$ is a Quillen adjunction with respect to the Quillen model structure of theorem~\ref{levelwise}. Observe that since the functor $N:\spc \longrightarrow \mathsf{Sp}^{\Sigma}(\cd,K')$ preserves fibrations and levelwise weak equivalences, proposition~\ref{Fibrations}, definition~\ref{leveleq} and the above commutative square imply that 
$$ N: \spcc \longrightarrow \mathsf{Sp}^{\Sigma}(\cd,K')\text{-}\mbox{Cat}$$
preserves fibrations and levelwise quasi-equivalences.

We now show that this adjunction induces a Quillen adjunction on the model structure of theorem~\ref{stable}. For this we use proposition~\ref{crit}. The functor $N$ clearly preserves trivial fibrations since they are the same in the model structures of theorems \ref{levelwise} and \ref{stable}.

Now, let $F:\ca \rightarrow \cb$ be a $Q$-fibration between Q-fibrant objects (see \ref{SfibS}). By proposition~\ref{Qfibfib}, $F$ is a fibration an so thus $N(F)$. Since the condition $\Omega)$ is clearly preserved by the functor $N$, proposition~\ref{Qfibfib} implies that the $\mathsf{Sp}^{\Sigma}(\cd,K')$-functor $N(F)$ is in fact a $Q$-fibration.

We now show that if the `reflexion' condition $R)$ is satisfied, the adjunction $(L_{cat},N)$ is in fact a Quillen equivalence.
For this, we verify conditions $a)$ and $b)$ of proposition~\ref{Ncrit}.
\begin{itemize}
\item[a)] Let $F:\ca \rightarrow \cb$ be a $\spc$-functor between $Q$-fibrant objects such that $N(F)$ is a stable quasi-equivalence in $\mathsf{Sp}^{\Sigma}(\cd,K')\text{-}\mbox{Cat}$. Since condition $\Omega)$ is clearly preserved by the functor $N$, proposition~\ref{weak=level} implies that $N(F)$ is in fact a levelwise quasi-equivalence. Since $\ca$ and $\cb$ are $Q$-fibrant, condition $a)$ of proposition~\ref{Ncrit} applied to the Quillen equivalence
$$ N: \spc \longrightarrow \mathsf{Sp}^{\Sigma}(\cd,K')$$
implies that $F$ satisfies condition $WE1)$. Now, condition $R)$ and the above commutative square imply that it satisfies condition $L2')$. By proposition~\ref{weak=level}, we conclude that $F$ is a stable quasi-equivalence.
\item[b)] Let $\ca$ be a cofibrant $\mathsf{Sp}^{\Sigma}(\cd,K')$-category and let us denote by $I$ its set of objects.
Notice, that in particular $\ca$ is cofibrant in $\mathsf{Sp}^{\Sigma}(\cd, K')^I\text{-}\mbox{Cat}$, with respect to the module structure of remark~\ref{nova1}. Since the adjunction $(L,N)$ is a strong monoidal Quillen equivalence (\ref{smwe}), remark~\ref{nova2} implies that condition $b)$ of proposition~\ref{Ncrit} is verified for the adjunction $(L^I,N^I)$. By remark~\ref{resolS}, the $\spc$-functor $\eta_{\ca}:\ca \rightarrow Q(\ca)$ is a fibrant resolution in $\spc^I$ and so we conclude that condition $b)$ is verified.
\end{itemize}
\end{proof}

\section{Examples of Quillen equivalences}\label{chapter5}
In this chapter we describe two situations, where the conditions of theorem~\ref{chave} are satisfied.

{\bf First situation:} We consider the relationship between the examples $1)$ and $2)$ described in~\ref{exampleC} and~\ref{exampleD}: recall from \cite[4.9]{Shipley}, that the inclusion $i: Ch_{\geq 0} \rightarrow Ch$ induces a strong monoidal Quillen equivalence
$$
\xymatrix{
\mathsf{Sp}^{\Sigma}(Ch, \mathbb{Z}[1]) \ar@<1ex>[d]^{\tau_{\geq 0}} \\
\mathsf{Sp}^{\Sigma}(Ch_{\geq 0}, \mathbb{Z}[1]) \ar@<1ex>[u]^i \,,
}
$$
with respect to the stable model structures. Moreover, the truncation functor $\tau_{\geq 0}$ preserves fibrations and levelwise weak equivalences. 
By remark~\ref{corleftadj}, we obtain a adjunction
$$
\xymatrix{
\mathsf{Sp}^{\Sigma}(Ch, \mathbb{Z}[1])\text{-}\mbox{Cat} \ar@<1ex>[d]^{\tau_{\geq 0}} \\
\mathsf{Sp}^{\Sigma}(Ch_{\geq 0}, \mathbb{Z}[1])\text{-}\mbox{Cat} \ar@<1ex>[u]^i \,.
}
$$
\begin{remark}\label{glob2}
Notice that we have the following commutative square
$$
\xymatrix{
\mathsf{dgcat} \ar[d]_{\tau_{\geq 0}} && \mathsf{Sp}^{\Sigma}(Ch, \mathbb{Z}[1])\text{-}\mbox{Cat} \ar[d]^{\tau_{\geq o}} \ar[ll]_-{Ev_0} \\
\mathsf{dgcat}_{\geq 0} && \mathsf{Sp}^{\Sigma}(Ch_{\geq 0}, \mathbb{Z}[1])\text{-}\mbox{Cat} \ar[ll]^-{Ev_0}\,.
}
$$
Finally, since the functor
$$ \tau_{\geq 0}: \mathsf{dgcat} \longrightarrow \mathsf{dgcat}_{\geq 0}$$
satisfies the `reflexion' condition $R)$ of theorem~\ref{chave}, the conditions of theorem~\ref{chave} are satisfied and so the previous adjunction is a Quillen equivalence.
\end{remark}

{\bf Second situation:} We consider the relationship between the examples $2)$ and $3)$ described in~\ref{exampleC} and~\ref{exampleD}: recall from \cite[4.4]{Shipley}, that we have a weak monoidal Quillen equivalence
$$
\xymatrix{
\mathsf{Sp}^{\Sigma}(s\mathbf{Ab}, \widetilde{\mathbb{Z}}(S^1)) \ar@<1ex>[d]^{\phi^{\ast}N} \\
\mathsf{Sp}^{\Sigma}(Ch_{\geq 0}, \mathbb{Z}[1]) \ar@<1ex>[u]^L\,,
}
$$
with respect to the stable model structures. Moreover the functor $\phi^{\ast}N$ preserves fibrations and levelwise weak equivalences (see the proof of \cite[4.4]{Shipley}). By proposition~\ref{Adjunction}, we have the adjunction
$$
\xymatrix{
\mathsf{Sp}^{\Sigma}(s\mathbf{Ab}, \widetilde{\mathbb{Z}}(S^1))\text{-}\mbox{Cat} \ar@<1ex>[d]^{\phi^{\ast}N} \\
\mathsf{Sp}^{\Sigma}(Ch_{\geq 0}, \mathbb{Z}[1])\text{-}\mbox{Cat} \ar@<1ex>[u]^{L_{cat}}\,.
}
$$
\begin{remark}\label{glob3}
Recall from \cite[4.4]{Shipley} that the ring map
$$ \phi: \mbox{Sym}(\mathbb{Z}[1]) \longrightarrow \cn$$ 
is the identity in degree zero. This implies that we have the following commutative square 
$$
\xymatrix{
s\mathbf{Ab}\text{-}\mbox{Cat} \ar[d]_N &&  \mathsf{Sp}^{\Sigma}(s\mathbf{Ab}, \widetilde{\mathbb{Z}}(S^1))\text{-}\mbox{Cat} \ar[ll]_-{Ev_0} \ar[d]^{\phi^{\ast}N} \\
\mathsf{dgcat}_{\geq 0} &&  \mathsf{Sp}^{\Sigma}(ch_{\geq 0}, \mathbb{Z}[1])\text{-}\mbox{Cat} \ar[ll]^-{Ev_0}\,.
}
$$
Finally, since the normalization functor (see section $5.2$ of~\cite{DGScat})
$$ N: s\mathbf{Ab}\text{-}\mbox{Cat} \longrightarrow \mathsf{dgcat}_{\geq 0}$$
satisfies the `reflexion' condition $R)$ of theorem~\ref{chave}, both conditions of theorem~\ref{chave} are satisfied and so the previous adjunction $(L_{cat}, \phi^{\ast}N)$ is a Quillen equivalence.
\end{remark}

\section{General Spectral algebra}\label{chapter6}
Let $R$ be a commutative symmetric ring spectrum of pointed simplicial sets (see \cite[I-1.3]{Schwede}). In this chapter, we construct a Quillen model structure on the category of small categories enriched over $R$-modules (see theorem~\ref{stable1}).

\begin{notation}
We denote by $(R\text{-}\mbox{Mod}, -\underset{R}{\wedge}-,R)$ the symmetric monoidal category of right $R$-modules (see \cite[I-1.5]{Schwede}). 
\end{notation}

Let us now recall some classical results of spectral algebra.
\begin{theorem}{\cite[III-3.2]{Schwede}}\label{model}
The category $R\text{-}\mbox{Mod}$ carries the following model structures:
\begin{itemize}
\item[-] The {\em projective level model structure}, in which the weak equivalences (resp. fibrations) are those morphisms of $R$-modules which are projective level equivalences (resp. projective level fibrations) on the underlying symmetric spectra.
\item[-] The {\em projective stable model structure}, in which the weak equivalences (resp. fibrations) are those morphisms of $R$-modules which are projective stable equivalences (resp. projective stable fibrations) on the underlying symmetric spectra.
\end{itemize}
Moreover these model structures are proper, simplicial, cofibrantly generated, monoidal (with respect to the smash product over $R$) and satisfy the monoid axiom.
\end{theorem}
Now, let $f:P \rightarrow R$ be a morphism (\cite[I-1.4]{Schwede}) of commutative symmetric ring spectra. 
\begin{remark}\label{restr/ext}
We have at our disposal a restriction/extension of scalars adjunction
$$
\xymatrix{
R\text{-}\mbox{Mod} \ar@<1ex>[d]^{f^{\ast}}\\
P\text{-}\mbox{Mod}\ar@<1ex>[u]^{f_!} \,.
}
$$
\end{remark}
Given a $R$-module $M$, we define a $P$-module $f^{\ast}(M)$ as the same underlying symmetric spectrum as $M$ and with the $P$-action given by the composite
$$(f^{\ast}M)\wedge P = M\wedge P \stackrel{Id\wedge f}{\longrightarrow} M\wedge R \stackrel{\alpha}{\longrightarrow} M\,,$$
where $\alpha$ denotes the $R$-action on $M$.

The left adjoint functor $f_!$ takes an $P$-module $N$ to the $R$-module $N\underset{P}{\wedge}R$ with the $R$-action given by
$$f_!N \wedge R = N\underset{P}{\wedge}R \wedge R \stackrel{Id\wedge \mu}{\longrightarrow} N \underset{P}{\wedge} R\,,$$
where $\mu$ denotes the multiplication of $R$.

\begin{theorem}{\cite[III-3.4]{Schwede}}
The previous adjuncton $(f_!, f^{\ast})$ is a Quillen adjunction with respect to the model structures of theorem~\ref{model}.
\end{theorem}

\begin{remark}
In the adjunction $(f_!,f^{\ast})$, the functor $f_!$ is strong monoidal and the functor $f^{\ast}$ is lax monoidal (see definition~\ref{lax}).
\end{remark}
We are now interested in the particular case where $P$ is the sphere symmetric ring spectrum $\mathbb{S}$ (see~\cite[I-2.1]{Schwede}) and $f=i$ the unique morphism of ring spectra. We have the restriction/extension of scalars adjunction
$$
\xymatrix{
R\text{-}\mbox{Mod} \ar@<1ex>[d]^{i^{\ast}}\\
\mathsf{Sp}^{\Sigma} \ar@<1ex>[u]^{i_!}\,,
}
$$
where $\mathsf{Sp}^{\Sigma}$ denotes the category of symmetric spectra~\cite{HSS}.
\begin{notation}
In what follows and to simply the notation, we denote by $R\text{-}\mbox{Cat}$ the category $(R\text{-}\mbox{Mod})\text{-}\mbox{Cat}$ (see notation~\ref{catenrich}). In particular a $(R\text{-}\mbox{Mod})$-category will be denoted by $R$-category and a $(R\text{-}\mbox{Mod})$-functor by $R$-functor.
\end{notation}
Notice that by remark~\ref{corleftadj}, the previous adjunction $(i_!, i^{\ast})$ induces the adjunction
$$
\xymatrix{
R\text{-}\mbox{Cat} \ar@<1ex>[d]^{i^{\ast}} \\
\mathsf{Sp}^{\Sigma}\text{-}\mbox{Cat} \ar@<1ex>[u]^{i_!} \,,
}
$$
where $\mathsf{Sp}^{\Sigma}\text{-}\mbox{Cat}$ denotes the category of spectral categories (see~\cite[3.3]{Spectral}).

\subsection{Levelwise quasi-equivalences}
In this section, we will construct a cofibrantly generated Quillen model structure on $R\text{-}\mbox{Cat}$, whose weak equivalences are as follows:
\begin{definition}\label{deflevelR}
A $R$-functor $F:\ca \rightarrow \cb$ is a {\em levelwise quasi-equivalence} if the restricted spectral functor $i^{\ast}(F):i^{\ast}(\ca) \rightarrow i^{\ast}(\cb)$ is a levelwise quasi-equivalence on $\mathsf{Sp}^{\Sigma}\text{-}\mbox{Cat}$ (see \cite[4.1]{Spectral}).
\end{definition}
\begin{notation}
We denote by $\cw_L$ the class of levelwise quasi-equivalences on $R\text{-}\mbox{Cat}$.
\end{notation}
\begin{remark}
Recall from \cite[4.8]{Spectral} that the category $\mathsf{Sp}^{\Sigma}\text{-}\mbox{Cat}$ carries a cofibrantly generated Quillen model structure whose weak equivalences are precisely the levelwise quasi-equivalences.
\end{remark}
\begin{theorem}\label{levelR}
If we let $\cm$ be the category $\mathsf{Sp}^{\Sigma}\text{-}\mbox{Cat}$, endowed with the model structure of theorem \cite[4.8]{Spectral}, and $\cn$ be the category $R\text{-}\mbox{Cat}$, then the conditions of the lifting theorem~\cite[11.3.2]{Hirschhorn} are satisfied. Thus, the category $R\text{-}\mbox{Cat}$ admits a cofibrantly generated Quillen model structure whose weak equivalences are the levelwise quasi-equivalences.
\end{theorem}
\begin{remark}\label{Rprop2}
Since the Quillen model structure on $\mathsf{Sp}^{\Sigma}\text{-}\mbox{Cat}$ is right proper (see~\cite[4.19]{Spectral}), the same holds for the model structure on $R\text{-}\mbox{Cat}$ given by theorem~\ref{levelR}.
\end{remark}
\begin{remark}\label{Rprop3}
By remark~\ref{Rprop2} and corollary~\cite[4.16]{Spectral}, a $R$-category $\ca$ is fibrant if and only if for all objects $x,y \in \ca$ and $n\geq 0$, the simplicial set $i^{\ast}(\ca)(x,y)_n$ is fibrant.
\end{remark}
\subsection{Proof of theorem~\ref{levelR}}
We start by observing that the category $R\text{-}\mbox{Cat}$ is complete and cocomplete. Since the restriction functor 
$$i^{\ast} : R\text{-}\mbox{Cat} \longrightarrow \mathsf{Sp}^{\Sigma}\text{-}\mbox{Cat}$$
preserves filtered colimits and the domains and codomains of the elements of the generating (trivial) cofibrations in $\mathsf{Sp}^{\Sigma}\text{-}\mbox{Cat}$ are small (see sections $4$ and $5$ of~\cite{Spectral}), condition $(1)$ of the lifting theorem~\cite[11.3.2]{Hirschhorn} is verified.

We now prove condition $(2)$. Recall from \cite[4.5]{Spectral}, that the set $J$ of generating trivial cofibrations in $\mathsf{Sp}^{\Sigma}\text{-}\mbox{Cat}$ consists of a (disjoint) union $J'\cup J''$.

\begin{lemma}\label{celll}
$i_!(J')\text{-}\mbox{cell} \subseteq \cw_L$.
\end{lemma}
\begin{proof}
Since the class $\cw_L$ is stable under transfinite compositions it is enough to prove the following: let $$S: i_!U(F_m\Lambda[k,n]_+) \longrightarrow \ca$$ be a $R$-functor and consider the following pushout in $R\text{-}\mbox{Cat}$
$$
\xymatrix{
i_!U(F_m\Lambda[k,n]_+) \ar[r]^-S \ar[d]_-{i_!(A_{m,k,n})} \ar@{}[dr]|{\lrcorner} & \ca \ar[d]^P \\
i_!U(F_m\Delta[n]_+) \ar[r] & \cb \,.
}
$$
We need to show that $P$ belongs to $\cw_L$. Observe that the $R$-functor $i_!(A_{m,k,n})$ identifies with the $R$-functor
$$ U(i_!(F_m\Lambda[k,n]_+)) \longrightarrow U(i_!(F_m\Delta[n]_+))\,.$$
Since the morphisms 
$$ i_!(F_m\Lambda[k,n]_+) \longrightarrow i_!(F_m\Delta[n]_+)$$
are trivial cofibrations in the projective level model structure of theorem~\ref{model}, proposition~\ref{clef} and theorem~\ref{model} imply that the spectral functor $i^{\ast}(P)$ satisfies condition $L1)$. Since $P$ induces the identity map on objects, the spectral functor $i^{\ast}(P)$ satisfies condition $L2')$ and so $P$ belongs to $\cw_L$.
\end{proof}
\begin{lemma}\label{cell2}
$i_!(J'')\text{-}\mbox{cell}\subseteq \cw_L$.
\end{lemma}
\begin{proof}
Since the class $\cw_L$ is stable under transfinite compositions, it is enough to prove the following: let $$S:i_!\underline{\mathbb{S}}=\underline{R} \longrightarrow \ca$$ be a $R$-functor (where $\underline{R}$ is the category with one object and endomorphism ring $R$) and consider the following pushout in $R\text{-}\mbox{cat}$
$$
\xymatrix{
\underline{R} \ar[r]^-S \ar[d]_-{i_!(A_{\ch})} \ar@{}[dr]|{\lrcorner} & \ca \ar[d]^P \\
i_!\Sigma^{\infty}(\ch_+) \ar[r] & \cb\,.
}
$$
We need to show that $P$ belongs to $\cw_L$. We start by showing that the spectral functor $i^{\ast}(P)$ satisfies condition $L1)$. Observe that the proof of condition $L1)$ is entirely analogous to the proof of condition $L1)$ in proposition~\ref{cell1}: simply replace the sets of $\mathsf{Sp}^{\Sigma}(\cc,K)$-functors $J'$ and $J''$ by the spectral functors $i_!(A_{m,n,k})$ and $i_!(A_{\ch})$ and the set $I$ (of generating cofibrations in $\mathsf{Sp}^{\Sigma}(\cc,K)\text{-}\mbox{Cat}$) by the set $i_!(I)$. 
 
We now show that the spectral functor $i^{\ast}(P)$ satisfies condition $L2')$. Notice that the existence of the spectral functor
$$ \Sigma^{\infty}(\ch_+) \longrightarrow i^{\ast}i_!\Sigma^{\infty}(\ch_+)\,,$$
given by the unit of the adjunction $(i_!, i^{\ast})$, implies that the left vertical spectral functor in the commutative square $$
\xymatrix{
i^{\ast} \underline{R} \ar[r]^-{i^{\ast}S} \ar[d]_{i^{\ast}i_!(A_{\ch})} & i^{\ast}(\ca) \ar[d]^{i^{\ast}(P)} \\
i^{\ast}i_! \Sigma^{\infty}(\ch_+) \ar[r] & i^{\ast}(\cb)\,,
}
$$
satisfies condition $L2')$.
Notice also, that if we restrict ourselves to the objects of each category ($\mbox{obj}(-)$) in the previous diagram, we obtain a co-cartesian square
$$
\xymatrix{
\mbox{obj}(i^{\ast} \underline{R}) \ar[r] \ar[d] \ar@{}[dr]|{\lrcorner} & \mbox{obj}(i^{\ast}(\ca) \ar[d] \\
\mbox{obj}(i^{\ast}i_! \Sigma^{\infty}(\ch_+)) \ar[r] & \mbox{obj}(i^{\ast}(\cb))
}
$$
in $\mathbf{Set}$. These facts clearly imply that $i^{\ast}(P)$ satisfies condition $L2')$ and so we conclude that $P$ belongs to $\cw_L$.
\end{proof}
Lemmas~\ref{celll} and \ref{cell2} imply condition $(2)$ of theorem~\cite[11.3.2]{Hirschhorn} and so theorem~\ref{levelR} is proven.
\begin{lemma}\label{U}
The functor
$$ U: R\text{-}\mbox{Mod} \longrightarrow R\text{-}\mbox{Cat}\,\,\,\,\,\,(\mbox{see} \,\,\ref{funcU})$$
sends projective cofibrations to cofibrations.
\end{lemma}
\begin{proof}
The model structure of theorem~\ref{levelR} is cofibrantly generated and so any cofibrant object in $R\text{-}\mbox{Cat}$ is a rectract of a (possibly infinite) composition of pushouts along the generating cofibrations. Since the functor $U$ preserves retractions, colimits and send the generating projective cofibrations to (generating) cofibrations the lemma is proven.
\end{proof}

\subsection{Stable quasi-equivalences}
In this section, we consider the projective stable model structure on $R\text{-}\mbox{Mod}$ (see \ref{model}). We will construct a Quillen model structure on $\cR\text{-}\mbox{Cat}$, whose weak equivalences are as follows:
\begin{definition}\label{stableR}
A $R$-functor $F:\ca \rightarrow \cb$ is a {\em stable quasi-equivalence} if it is an weak equivalence in the sense of definition~\ref{we}.
\end{definition}
\begin{notation}
We denote by $\cw_S$ the class of stable quasi-equivalences.
\end{notation}
\begin{proposition}\label{coincide}
A $R$-functor $F:\ca \rightarrow \cb$ is a stable quasi-equivalence if and only if the restricted spectral functor $i^{\ast}(F): i^{\ast}(\ca) \rightarrow i^{\ast}(\cb)$ is a stable quasi-equivalence in $\mathsf{Sp}^{\Sigma}\text{-}\mbox{Cat}$ (see~\cite[5.1]{Spectral}).
\end{proposition}
\begin{proof}
We start by noticing that the definition of stable equivalence in $R\text{-}\mbox{Mod}$ (see~\ref{model}), implies that a $R$-functor $F$ satisfies condition $WE1)$ if and only if the spectral functor $i^{\ast}(F)$ satisfies condition $S1)$ of definition \cite[5.1]{Spectral}. Now, let $\ca$ be a $R$-category. Observe that, since the functor 
$$i_!:\mathsf{Sp}^{\Sigma} \longrightarrow R\text{-}\mbox{Mod}$$
is strong monoidal, the category $[\ca]$ (see~\ref{[]}) naturally identifies with $[i^{\ast}(\ca)]$ and so $F$ satisfies condition $WE2)$ if and only if $i^{\ast}(F)$ satisfies condition $S2)$. 

This proves the proposition. 
\end{proof}
\subsection{$Q$-functor}
We now construct a functor
$$Q:R\text{-}\mbox{Cat} \longrightarrow R\text{-}\mbox{Cat}$$
and a natural transformation $\eta:Id \rightarrow Q$, from the identity functor on $R \text{-}\mbox{Cat}$ to $Q$.

\begin{remark} Recall from \cite[5.5]{Spectral} that we have at our disposal a set $\{A_{m,k,n}\} \cup U(K)$ of spectral functors used in the construction of a $Q$-functor in $\mathsf{Sp}^{\Sigma}\text{-}\mbox{Cat}$.
\end{remark}
\begin{definition}\label{defomega}
Let $\ca$ be a $R$-category. The functor $Q: R\text{-}\mbox{Cat} \longrightarrow R\text{-}\mbox{Cat}$ is obtained by applying the small object argument, using the set $ \{ i_!(A_{m,k,n})\} \cup i_!U(K)$, to factorize the $R$-functor
$$ \ca \longrightarrow \bullet\,,$$
where $\bullet$ denotes the terminal object in $R\text{-}\mbox{Cat}$.
\end{definition}

\begin{remark}\label{Rkomega}
We obtain in this way a functor $Q$ and a natural transformation $\eta:Id \rightarrow Q$. Notice that the set $ \{i_!(A_{m,k,n})\} \cup i_!U(K) $ identifies naturally with the set $  \{ U(i_!(A_{m,n,k}))\} \cup U(i_!(K))$. Notice also that $Q(\ca)$ has the same objects as $\ca$, and the R.L.P with respect to the set $  \{ U(i_!(A_{m,n,k}))\} \cup U(i_!(K))$. By adjunction and remark~\cite[5.7]{Spectral} this is equivalent to the following fact: the spectral category $i^{\ast}Q(\ca)$ satisfies the following condition
\begin{itemize}
\item[$\Omega)$] for all objects $x,y \in i^{\ast}Q(\ca)$ the symmetric spectra $(i^{\ast}Q)(x,y)$ is an $\Omega$-spectrum.
\end{itemize}
\end{remark}
\begin{proposition}\label{weak1}
Let $\ca$ be a $R$-category. The $R$-functor
$$ \eta_{\ca}:\ca \longrightarrow Q(\ca)$$
is a stable quasi-equivalence (\ref{stableR}).
\end{proposition}
\begin{proof}
We start by observing that since the elements of the set $\{A_{m,n,k}\} \cup K$ are trivial cofibrations in the projective stable model structure on $\mathsf{Sp}^{\Sigma}$, the elements of the set $ \{ i_!(A_{m,n,k})\} \cup i_!(K)$ are trivial cofibrations in the stable model structure on $R\text{-}\mbox{Mod}$ (see \ref{model}). By theorem~\ref{model} this model structure is monoidal and satisfies the monoid axiom. This implies, by proposition~\ref{clef}, that the spectral functor $\eta_{\ca}$ satisfies condition $WE1)$. Since $Q(\ca)$ and $\ca$ have the same set of objects condition $WE2')$ is automatically verified and so we conclude that $\eta_{\ca}$ belongs to $\cw_S$.
\end{proof}

\subsection{Main theorem}
\begin{definition}
A $R$-functor $F:\ca \rightarrow \cb$ is:
\begin{itemize}
\item[-] a {\em $Q$-weak equivalence} if $Q(F)$ is a levelwise quasi-equivalence (see~\ref{deflevelR}).
\item[-] a {\em cofibration} if it is a cofibration in the model structure of theorem~\ref{levelR} .
\item[-] a {\em $Q$-fibration} if it has the R.L.P. with respect to all cofibrations which are $Q$-weak equivalences.
\end{itemize}
\end{definition}
\begin{lemma}\label{coincide2}
A $R$-functor $F:\ca \rightarrow \cb$ is a $Q$-weak equivalence if and only if it is a stable quasi-equivalence (\ref{stableR}).
\end{lemma}
\begin{proof}
We have at our disposal a commutative square
$$
\xymatrix{
\ca \ar[r]^{\eta_{\ca}} \ar[d]_F & Q(\ca) \ar[d]^{Q(F)} \\
\cb \ar[r]_{\eta_{\cb}} & Q(\cb) 
}
$$
where the $R$-functors $\eta_{\ca}$ and $\eta_{\cb}$ are stable quasi-equivalences by proposition~\ref{weak1}. If we apply the restriction functor $i^{\ast}$ to the above square, we obtain
$$
\xymatrix{
i^{\ast}\ca \ar[d]_-{i^{\ast}(F)} \ar[r]^-{i^{\ast}(\eta_{\ca})} & i^{\ast}Q(\ca) \ar[d]^-{i^{\ast}Q(F)} \\
i^{\ast}\cb \ar[r]_-{i^{\ast}(\eta_{\cb})} & i^{\ast}Q(\cb)\,,
}
$$
where the spectral functors $i^{\ast}(\eta_{\ca})$ and $i^{\ast}(\eta_{\cb})$ are stable quasi-equivalences in $\mathsf{Sp}^{\Sigma}\text{-}\mbox{Cat}$ (see~\ref{coincide}) and the spectral categories $i^{\ast}Q(\ca)$ and $i^{\ast}Q(\cb)$ satisfy condition $\Omega)$ (see~\ref{Rkomega}). Now, we conclude: $F$ is a weak equivalence if and only if $i^{\ast}(F)$ is a stable quasi-equivalence (\ref{coincide}); if and only if $i^{\ast}Q(F)$ is a stable quasi-equivalence; if and only if $i^{\ast}Q(F)$ is a levelwise quasi-equivalence (theorem~\cite[5.11-(A1)]{Spectral}); if and only if $F$ is a $Q$-weak equivalence.
\end{proof}
\begin{theorem}\label{stable1}
The category $R\text{-}\mbox{Cat}$ admits a right proper Quillen model structure whose weak equivalences are the stable quasi-equivalences (\ref{stableR}) and the cofibrations those of theorem~\ref{levelR}.
\end{theorem}
\begin{proof}
The proof will consist on verifying the conditions of theorem~\ref{modif}. We consider for $\cm$ the Quillen model structure of theorem~\ref{levelR} and for $Q$ and $\eta$, the functor and natural transformation defined in \ref{defomega}. The Quillen model structure of theorem~\ref{levelR} is right proper (\ref{Rprop2}) and by lemma~\ref{coincide2}, the $Q$-weak equivalences are precisely the stable quasi-equivalences. We now verify conditions (A1), (A2) and (A3):
\begin{itemize}
\item[(A1)] Let $F:\ca \rightarrow \cb$ be a levelwise quasi-equivalence. We have the following commutative square
$$
\xymatrix{
\ca \ar[d]_F \ar[r]^{\eta_{\ca}} & Q(\ca) \ar[d]^{Q(F)} \\
\cb \ar[r]_{\eta_{\cb}} & Q(\cb)
}
$$
with $\eta_{\ca}$ and $\eta_{\cb}$ stable quasi-equivalences. If we apply the restriction functor $i^{\ast}$ to the above square, we obtain 
$$
\xymatrix{
i^{\ast}\ca \ar[d]_{i^{\ast}(F)} \ar[r]^-{i^{\ast}(\eta_{\ca})} & i^{\ast}Q(\ca) \ar[d]^{i^{\ast}Q(F)} \\
i^{\ast}\cb \ar[r]_-{i^{\ast}(\eta_{\cb})} & i^{\ast}Q(\cb)\,,
}
$$
where $i^{\ast}(F)$ is a levelwise quasi-equivalence (\ref{deflevelR}), $i^{\ast}(\eta_{\ca})$ and $i^{\ast}(\eta_{\cb})$ are stable quasi-equivalences (\ref{coincide}) and $i^{\ast}Q(\ca)$ and $i^{\ast}Q(\cb)$ satisfying the condition $\Omega)$ (see~\ref{Rkomega}). Now, condition (A1) of theorem \cite[5.11]{Spectral}, implies condition (A1).
\item[(A2)] We now show that for every $R$-category $\ca$, the $R$-functors
$$ \eta_{Q(\ca)}, Q(\eta_{\ca}):Q(\ca) \longrightarrow QQ(\ca)$$
are stable quasi-equivalences. If we apply the restriction functor $i^{\ast}$ to the above $R$-functors, we obtain stable quasi-equivalences between spectral categories which satisfy condition $\Omega)$. Now, condition (A2) of theorem~\cite[5.11]{Spectral} implies condition (A2).
\item[(A3)] We start by observing that if $P:\ca \longrightarrow \cd$ is a $Q$-fibration, then for all objects $x,y \in \cc$, the morphism
$$ P(x,y):\cc(x,y) \longrightarrow \cd(Px,Py)$$
is a fibration in the projective stable model structure on $R\text{-}\mbox{Mod}$ (see \ref{model}). In fact, by proposition~\ref{U}, the functor
$$U: R\text{-}\mbox{Mod} \longrightarrow R\text{-}\mbox{Cat}$$
sends projective cofibrations to cofibrations. Since clearly it sends stable equivalences to stable quasi-equivalences the claim follows.

Now consider the following diagram in $R\text{-}\mbox{Cat}$
$$
\xymatrix{
\ca \underset{Q(\ca)}{\times}\cb \ar@{}[dr]|{\ulcorner} \ar[r] \ar[d] & \cb \ar[d]^P \\
\ca \ar[r]_{\eta_{\ca}} & Q(\ca)\,,
}
$$
with $P$ a $Q$-fibration. If we apply the restriction functor $i^{\ast}$ to the previous commutative square, we obtain
$$
\xymatrix{
i^{\ast}\ca \underset{i^{\ast}Q(\ca)}{\times} i^{\ast}\cb \ar@{}[dr]|{\ulcorner} \ar[r] \ar[d] & i^{\ast}(\cb) \ar[d]^{i^{\ast}(P)} \\
i^{\ast}\ca \ar[r]_{i^{\ast}(\eta_{\ca})} & i^{\ast}Q(\ca)
}
$$
a cartesian square, with $i^{\ast}(\eta_{\ca})$ a stable quasi-equivalence. Notice also that $i^{\ast}(P)$ is such that for all objects $x,y \in i^{\ast}(\cb)$
$$ i^{\ast}(P)(x,y): i^{\ast}(\cb)(x,y) \longrightarrow i^{\ast}Q(\ca)$$
is a fibration in the projective stable model structure on $\mathsf{Sp}^{\Sigma}$.
Now, condition (A3) of theorem~\cite[5.11]{Spectral} implies condition (A3).
\end{itemize}
The theorem is now proved.
\end{proof}

\begin{proposition}\label{Sfibrant}
A $R$-category $\ca$ is fibrant with respect to the model structure of theorem~\ref{stable1} if and only if for all objects $x,y \in \ca$, the symmetric spectrum $(i^{\ast}\ca)(x,y)$ is an $\Omega$-spectrum.
\end{proposition}
\begin{notation}
We denote these fibrant $R$-categories by $Q$-fibrant.
\end{notation}
\begin{proof}
By corollary $A.5$, $\ca$ is fibrant, with respect to theorem~\ref{stable1} if and only if it is fibrant for the structure of theorem~\ref{levelR} (see remark~\ref{Rprop3}) and the $R$-functor $\eta_{\ca}:\ca \rightarrow Q(\ca)$ is a levelwise quasi-equivalence. Observe that 
$$i^{\ast}(\eta_{\ca}):i^{\ast}(\ca) \rightarrow i^{\ast}Q(\ca)$$
 is a levelwise quasi-equivalence if and only if for all objects $x,y \in \ca$, the morphism of symmetric spectra
$$ i^{\ast}(\eta_{\ca}(x,y)):i^{\ast}\ca(x,y) \longrightarrow i^{\ast}Q(\ca)(x,y)$$
is a level equivalence. Since for all objects $x,y \in \ca$, the symmetric spectrum $i^{\ast}Q(\ca)(x,y)$ is an $\Omega$-spectrum (see \ref{Rkomega}) we have the following commutative diagrams (for all $n \geq 0$)
$$
\xymatrix{
i^{\ast} \ca(x,y)_n \ar[d] \ar[r]^-{\widetilde{\delta_n}} & \Omega(i^{\ast}\ca(x,y)_{n+1}) \ar[d] \\
i^{\ast} Q(\ca)(x,y)_n \ar[r]_-{\widetilde{\delta_n}} &  \Omega (i^{\ast}Q(\ca)(x,y)_{n+1})\,,
}
$$  where the bottom and vertical arrows are weak equivalences of pointed simplicial sets. This implies that
$$ \widetilde{\delta}_n: i^{\ast}\ca(x,y)_n \longrightarrow \Omega (i^{\ast}\ca(x,y)_{n+1})\,,\,\,n \geq 0$$
is a weak equivalence of pointed simplicial sets. In conclusion, for all objects $x,y \in \ca$, the symmetric spectrum $ i^{\ast}\ca(x,y)$ is an $\Omega$-spectrum.
\end{proof}
\begin{remark}\label{condb}
Remark~\ref{Rkomega} and proposition~\ref{Sfibrant} imply that $\eta_{\ca}:\ca \rightarrow Q(\ca)$ is a (functorial) fibrant replacement of $\ca$ in the model structure of theorem~\ref{stable1}. In particular $\eta_{\ca}$ induces the identity map on the set of objects.
\end{remark}

\begin{proposition}\label{fibQfib}
Let $F:\ca \rightarrow \cb$ be a fibration (in the model structure of theorem~\ref{levelR}) between $Q$-fibrant objects (see~\ref{Sfibrant}). Then $F$ is a $Q$-fibration.
\end{proposition}
\begin{proof}
We need to show (by theorem~\ref{Qfib}) that the following commutative square
$$
\xymatrix{
\ca \ar[d]_F \ar[r]^-{\eta_{\ca}} & Q(\ca) \ar[d]^{Q(F)} \\
\cb \ar[r]_-{\eta_{\cb}} & Q(\cb)
}
$$
is homotopy cartesian (in the model structure of theorem~\ref{levelR}). By theorem~\ref{levelR} this is equivalent to the fact that the following commutative square in $\mathsf{Sp}^{\Sigma}\text{-}\mbox{Cat}$
$$
\xymatrix{
i^{\ast}(\ca) \ar[d]_{i^{\ast}(F)} \ar[r]^-{i^{\ast}(\eta_{\ca})} & Q(\ca) \ar[d]^{i^{\ast}Q(F)} \\
i^{\ast}(\cb) \ar[r]_-{i^{\ast}(\eta_{\cb})} & i^{\ast}Q(\cb)
}
$$
is homotopy cartesian. Notice that since $\ca$ and $\cb$ are $Q$-fibrant objects, the spectral functors $i^{\ast}(\eta_{\ca})$ and $i^{\ast}(\eta_{\cb})$ are levelwise quasi-equivalences. Since $i^{\ast}(F)$ is a fibration, the proposition is proven. 
\end{proof}
Let $f:P\rightarrow R$ be a morphism of commutative symmetric ring spectra. Notice that the restriction/extension of scalars adjunction (\ref{restr/ext}) induces the natural adjunction
$$
\xymatrix{
R\text{-}\mbox{Cat} \ar@<1ex>[d]^{f^{\ast}}\\
P\text{-}\mbox{Cat}\ar@<1ex>[u]^{f_!} \,.
}
$$
\begin{proposition}\label{Top}
The previous adjunction $(f_!,f^{\ast})$ is a Quillen adjunction with respect to the model structures of theorems~\ref{levelR} and \ref{stable1}.
\end{proposition}
\begin{proof}
We start by considering the model structure of theorem~\ref{levelR}. Notice that the commutative diagram
$$
\xymatrix{
P \ar[rr]^f  && R \\
 & \mathbb{S} \ar[lu]^{i_P} \ar[ur]_{i_R} & 
}
$$
induces a {\em strictly} commutative diagram of categories
$$
\xymatrix{
P\text{-}\mbox{Cat} \ar[dr]_{i_P^{\ast}}  && R\text{-}\mbox{Cat} \ar[ll]_{f^{\ast}} \ar[dl]^{i_R^{\ast}} \\
 & \mathsf{Sp}^{\Sigma}\text{-}\mbox{Cat} & \,.
}
$$
By theorem~\ref{levelR}, we conclude that $f^{\ast}$ preserves (and reflects) levelwise quasi-equivalences and fibrations.

Now, to show the proposition with respect to the model structure of theorem~\ref{stable1}, we use proposition~\ref{crit}.
\begin{itemize}
\item[a)] Clearly $f^{\ast}$ preserves trivial fibrations, since these are the same in the Quillen model structures of theorems~\ref{levelR} and \ref{stable1}.
\item[b)] Let $F:\ca \rightarrow \cb$ be a $Q$-fibration in $R\text{-}\mbox{Cat}$ between fibrant objects (see~\ref{Sfibrant}). By theorem~\ref{Qfib}, $F$ is a fibration and so by $a)$ so it is $f^{\ast}(F)$. Notice that proposition~\ref{Sfibrant} and the above strictly commutative diagram of categories imply that $f^{\ast}$ preserves $Q$-fibrant objects. Now, we conclude by proposition~\ref{fibQfib} that $f^{\ast}(F)$ is in fact a $Q$-fibration.
\end{itemize}
The proposition is proven.
\end{proof}
\section{Eilenberg-Mac Lane spectral algebra}\label{chapter7}
In this chapter we prove that in the case of the Eilenberg-Mac Lane ring spectrum $H\mathbb{Z}$, the model structure of theorem~\ref{stable1} is Quillen equivalent to the one described in chapter~\ref{chapter5} for the case of simplicial $\mathbb{Z}$-modules (see proposition~\ref{glob4}).
Recall from \cite[4.3]{Shipley} that we have the following strong monoidal Quillen equivalence
$$
\xymatrix{
\mathsf{Sp}^{\Sigma}(s\mathbf{Ab}, \widetilde{\mathbb{Z}}(S^1)) \ar@<1ex>[d]^U \\
H\mathbb{Z}\text{-}\mbox{Mod} \ar@<1ex>[u]^Z \,,
}
$$
where $U$ is the forgetful functor. By remark~\ref{corleftadj}, it induces the following adjunction
$$
\xymatrix{
\mathsf{Sp}^{\Sigma}(s\mathbf{Ab}, \widetilde{\mathbb{Z}}(S^1))\text{-}\mbox{Cat} \ar@<1ex>[d]^U \\
H\mathbb{Z}\text{-}\mbox{Cat} \ar@<1ex>[u]^Z \,,
}
$$
denoted by the same functors.
\begin{proposition}\label{glob4}
The previous adjunction is a Quillen equivalence, when the category $\mathsf{Sp}^{\Sigma}(s\mathbf{Ab}, \widetilde{\mathbb{Z}}(S^1))\text{-}\mbox{Cat}$ is endowed with the model structure of theorem~\ref{stable} and $H\mathbb{Z}\text{-}\mbox{Cat}$ is endowed with the model structure of theorem~\ref{stable1}.
\end{proposition}
\begin{proof}
We start by showing that the previous adjunction is a Quillen adjunction, when $\mathsf{Sp}^{\Sigma}(s\mathbf{Ab})\text{-}\mbox{Cat}$ is endowed with the model structure of theorem~\ref{levelwise} and $H\mathbb{Z}\text{-}\mbox{Cat}$ is endowed with the model structure of theorem~\ref{levelR}.
Notice that we have the following commutative diagram
$$
\xymatrix{
\mathsf{Sp}^{\Sigma}(s\mathbf{Ab}, \widetilde{\mathbb{Z}}(S^1))\text{-}\mbox{Cat} \ar[d]_U \ar[rr]^{Ev_0} && s\mathbf{Ab}\text{-}\mbox{Cat} \ar[d]^U \\
H\mathbb{Z}\text{-}\mbox{Cat} \ar[r] & \mathsf{Sp}^{\Sigma}\text{-}\mbox{Cat} \ar[r]_-{Ev_0} & s\mathbf{Set}\text{-}\mbox{Cat}\,.
}
$$
Since the forgetful functor $U:\mathsf{Sp}^{\Sigma}(s\mathbf{Ab}, \widetilde{\mathbb{Z}}(S^1)) \rightarrow H\mathbb{Z}\text{-}\mbox{Cat}$ preserves fibrations and levelwise weak equivalences the claim follows.

We now show that it is a Quillen adjunction with respect to the model structures described in the proposition. Clearly $U$ preserves trivial fibrations. Since fibrations between $Q$-fibrant objects are $Q$-fibrations, see lemmas \ref{fibQfib} and \ref{Qfibfib}, these are also preserved by $U$. By proposition~\ref{crit}, we conclude that it is a Quillen adjunction.

Now, we show that $U$ is a Quillen equivalence. For this we verify conditions $a)$ and $b)$ of proposition~\ref{Ncrit}.
\begin{itemize}
\item[a)] Let $R:\ca \rightarrow \cb$ be a $\mathsf{Sp}^{\Sigma}(s\mathbf{Ab},\widetilde{\mathbb{Z}}(S^1))$-functor between $Q$-fibrant objects such that $U(R): U(\ca) \rightarrow U(\cb)$ is a stable quasi-equivalence. Since $\ca$ and $\cb$ are $Q$-fibrant, condition $a)$ of proposition~\ref{Ncrit} applied to the Quillen equivalence
$$ U: \mathsf{Sp}^{\Sigma}(s\mathbf{Ab}, \widetilde{\mathbb{Z}}(S^1)) \longrightarrow H\mathbb{Z}\text{-}\mbox{Mod}$$
implies that $R$ satisfies condition $WE1)$. Now, since the functor 
$$U:s\mathbf{Ab}\text{-}\mbox{Cat} \rightarrow s\mathbf{Set}\text{-}\mbox{Cat}$$
satisfies the `reflexion' condition $R)$ of theorem~\ref{chave}, the $\mathsf{Sp}^{\Sigma}(s\mathbf{Ab}, \widetilde{\mathbb{Z}}(S^1))$-functor $R$ satisfies condition $L2')$ and so we conclude that it is a stable quasi-equivalence.
\item[b)] Let $\ca$ be a cofibrant $H\mathbb{Z}$-category and let us denote by $I$ its set of objects. 
Notice, that in particular $\ca$ is cofibrant in $(H\mathbb{Z}\text{-}\mbox{Mod})^I\text{-}\mbox{Cat}$, with respect to the module structure of remark~\ref{nova1}. Now, since the adjunction $(L,F)$ is a strong monoidal Quillen equivalence (\ref{smwe}), remark~\ref{condb} implies that condition $b)$ of proposition~\ref{Ncrit} is verified for the adjunction $(L^I,F^I)$. By remark~\ref{nova2}, the $\spc$-functor $\eta_{\ca}:\ca \rightarrow Q(\ca)$ is a fibrant resolution in $(H\mathbb{Z}\text{-}\mbox{Mod})^I\text{-}\mbox{Cat}$ and so we conclude that condition $b)$ is verified.
\end{itemize}
\end{proof}
\section{Global picture}\label{global}
Notice that remarks \ref{glob1}, \ref{glob2}, \ref{glob3} and proposition~\ref{glob4} furnishes us the following (four steps) zig-zag of Quillen equivalences:
$$
\xymatrix{
H\mathbb{Z}\text{-}\mbox{Cat} \ar@<-1ex>[d]_Z \\
\mathsf{Sp}^{\Sigma}(s\mathbf{Ab}, \widetilde{\mathbb{Z}}(S^1))\text{-}\mbox{Cat} \ar@<-1ex>[u]_U \ar@<1ex>[d]^{\phi^{\ast}N} \\
\mathsf{Sp}^{\Sigma}(Ch_{\geq 0}, \mathbb{Z}[1])\text{-}\mbox{Cat} \ar@<1ex>[u]^{L_{cat}} \ar@<-1ex>[d]_i \\
\mathsf{Sp}^{\Sigma}(Ch, \mathbb{Z}[1])\text{-}\mbox{Cat} \ar@<-1ex>[u]_{\tau_{\geq 0}} \ar@<1ex>[d]^{Ev_0} \\
\mathsf{dgcat} \ar@<1ex>[u]^{F_0} \,.
}
$$
\begin{notation}
We denote by 
$$\mathbb{H}: \mathsf{Ho}(\mathsf{dgcat}) \stackrel{\sim}{\longrightarrow} \mathsf{Ho}(H\mathbb{Z}\text{-}\mbox{Cat})$$
and by
$$ \Theta: \mathsf{Ho}(H\mathbb{Z}\text{-}\mbox{Cat}) \stackrel{\sim}{\longrightarrow} \mathsf{Ho}(\mathsf{dgcat})$$
the composed derived equivalences.
\end{notation}
\begin{remark}
If, in the previous diagram, we restrict ourselves to enriched categories with only one object, we recover the zig-zag of Quillen equivalences constructed by Shipley in \cite[2.10]{Shipley} and \cite[4.9]{Shipley}.
\end{remark}
\section{$THH$ and $TC$ for DG categories}\label{THH}
In this chapter, we use the equivalence of the previous chapter to define a topological Hochschild and cyclic homology theory for dg categories.

Let us start by recalling from section $3$ of \cite{Mandell}, the construction of topological Hochschild homology ($THH$) and topological cyclic homology ($TC$) in the context of spectral categories. 

{\bf Construction:}
Let $\ci$ be the category whose objects are the finite sets $\mathbf{n}=\{1, \ldots, n\}$ (including ${\bf 0}=\{ \}$), and whose morphisms are the injective maps. For a symmetric spectrum $T$ (of pointed simplicial sets), we write $T_n$ for the $n$-th pointed simplicial set. The association $\mathbf{n} \mapsto \Omega^n|T_n|$ extends to a functor from $\ci$ to spaces, where $|-|$ denotes the geometric realization.

More generally, given symmetric spectra $T^0, \ldots, T^q$ and a space $X$, we obtain a functor from $\ci^{q+1}$ to spaces, which sends $(\mathbf{n}_0, \ldots, \mathbf{n}_q)$ to 
$$\Omega^{n_0+\cdots+n_q}(|T^q_{n_q}\wedge \cdots \wedge T^0_{n_0}|\wedge X)\,.$$

Now, let $\ca$ be a spectral category. Let $V(\ca,X)_{n_0, \ldots ,n_q}$ be the functor from $\ci^{q+1}$ to spaces defined by
$$\Omega^{n_0+\cdots+n_q}(\bigvee |\ca(c_{q-1},c_q)_{n_q} \wedge \cdots \wedge \ca(c_0,c_1)_{n_1}\wedge \ca(c_q,  c_0)_{n_0}| \wedge X)\,,$$
where the sum $\bigvee$ is taken over the $(q+1)$-tuples $(q_0, \cdots, q_n)$ of objects of $\ca$.

Now, define
$$ THH_q(\ca)(X):= \mbox{hocolim}_{\ci^{q+1}}V(\ca,X)_{n_0, \cdots, n_q}\,.$$
This construction assembles into a simplicial space and we write $THH(\ca)(X)$ for its geometric realization. We obtain in this way a functor $THH(\ca)(X)$ in the variable $X$. If we restrict ourselves to the spheres $S^n$, we obtain finally the symmetric spectrum of {\em topological Hochschild homology} of $\ca$ denoted by $THH(\ca)$. 

For the definition of topological cyclic homology involving the `cyclotomic' structure on $THH(\ca)$, the author is invited to consult section $3$ of \cite{Mandell}. We denote by $TC(\ca)$ the spectrum of {\em topological cyclic homology} of $\ca$.
\begin{remark}
By theorem~\cite[4.9]{Mandell} (and proposition~\cite[3.8]{Mandell}), the functors
$$THH, TC: \mathsf{Sp}^{\Sigma}\text{-}\mbox{Cat} \longrightarrow \mathsf{Sp}\,,$$
send quasi-equivalences to weak equivalences and so descend to the homotopy categories.
\end{remark}
Now, let $i: \mathbb{S} \rightarrow H\mathbb{Z}$ be the unique morphism of commutative symmetric ring spectra. By proposition~\ref{Top}, we have the restriction/extension of scalars Quillen adjunction
$$
\xymatrix{
H\mathbb{Z}\text{-}\mbox{Cat} \ar@<1ex>[d]^{i^{\ast}} \\
\mathsf{Sp}^{\Sigma}\text{-}\mbox{Cat} \ar@<1ex>[u]^{i_!}\,.
}
$$
Notice that the equivalence of the previous chapter
$$ \mathbb{H}:\mathsf{Ho}(\mathsf{dgcat}) \stackrel{\sim}{\longrightarrow} \mathsf{Ho}(H\mathbb{Z}\text{-}\mbox{Cat})$$
and the previous adjunction furnishes us well-defined topological Hochschild and cyclic homology theories 
$$
\xymatrix{
THH, TC: \mathsf{Ho}(\mathsf{dgcat}) \ar[r]^-{\mathbb{H}}_-{\sim} \ar@{-->}[drr]& \mathsf{Ho}(H\mathbb{Z}\text{-}\mbox{Cat})\ar[r]^{i^{\ast}} & \mathsf{Ho}(\mathsf{Sp}^{\Sigma}\text{-}\mbox{Cat}) \ar[d]^{THH,TC} \\
&  & \mathsf{Ho}(\mathsf{Sp})
}
$$
for dg categories (up to quasi-equivalence). Since the $THH$ and $TC$ homology carries much more torsion information than its Hochschild and cyclic homology, we obtain in this way a much richer invariant of dg categories and of (algebraic) varieties.

\begin{example}\label{exTop}
Let $X$ be a smooth projective algebraic variety. We associate to $X$ a dg model, i.e. a dg category $\cd^b_{dg}(\mbox{coh}(X))$, well defined up to isomorphism in $\mathsf{Ho}(\mathsf{dgcat})$, such that the triangulated category $H^0(\cd^b_{dg}(\mbox{coh}(X)))$ is equivalent the bounded derived category of quasi-coherent sheaves on $X$. For example, for $\cd^b_{dg}(\mbox{coh}(X))$, we can take the dg category of left bounded complexes of injective $\co_X$-modules whose homology is bounded and coherent. We then define the $THH$ and $TC$ of $X$ as the $THH$ and $TC$ of the spectral category $i^{\ast}(\mathbb{H}(\cd^b_{dg}(\mbox{coh}(X))))$.
\end{example}
\subsection{$THH$ and $TC$ as additive invariants}
In this section, we use freely the language of derivators, see chapter $1$ of~\cite{CN} for the main notions.
Notice that the constructions of the previous chapters furnishes us morphisms of derivators
$$ THH, TC : \mathsf{HO}(\dgcat)_{QE} \longrightarrow \mathsf{HO}(\mathsf{Sp})\,,$$
from the derivator $\mathsf{HO}(\dgcat)_{QE}$ associated with the Quillen model structure \cite[1.8]{These} to the derivator associated to the category of spectra. We now show that $THH$ and $TC$ descent to the derivator $\mathsf{HO}(\dgcat)$ associated with the Morita model structure \cite[2.27]{These}. For this, we start by recalling some properties of $THH$ and $TC$ in the context of spectral categories.
\begin{definition}{(see section $4$ from \cite{Mandell})}
A spectral functor $F:\cc \rightarrow \cd$ is a {\em Morita equivalence} if:
\begin{itemize}
\item[-] for all objects $x, y \in \cc$, the morphism
$$ F(x,y):\cc(x,y) \longrightarrow \cd(Fx,Fy)$$
is a stable weak equivalence in $\mathsf{Sp}^{\Sigma}$ and 
\item[-] the smallest thick (i.e. closed under direct factors) triangulated subcategory of $\mathsf{Ho}(\cd\text{-}\mbox{Mod})$ (see \ref{modules}) generated by the image of $F$ is equivalent to the full subcategory of compact objects in $\mathsf{Ho}(\cd\text{-}\mbox{Mod})$.
\end{itemize}
\end{definition}
\begin{remark}\label{Morinv}
Observe that theorem \cite[4.12]{Mandell} and proposition \cite[3.8]{Mandell} imply that if a spectral functor $F$ is a Morita equivalence, then $THH(F)$ and $TC(F)$ are weak equivalences.
\end{remark}
Now, let us return to the context of dg categories.
\begin{proposition}\label{Mor-Sp}
Let $F:\ca \rightarrow \cb$ be a Morita dg functor (see section $2.5$ of \cite{These}). Then the associated spectral functor $i^{\ast}\mathbb{H}(F)$ is a Morita equivalence.
\end{proposition}
\begin{proof}
By propositions \cite[2.35]{These} and \cite[2.14]{These}, $F$ is a Morita dg functor if and only if $F_f$ 
$$ 
\xymatrix{
\ca \ar[r]^F \ar[d]_{j_{\ca}} & \cb \ar[d]^{j_{\cb}} \\
\ca_f \ar[r]_{F_f} & \cb_f
}
$$
is a quasi-equivalence, where $(-)_f$ denotes a functorial fibrant resolution functor in the model category \cite[2.27]{These}. Notice that if we apply the composed functor $i^{\ast}\mathbb{H}$ to the above diagram, we obtain a diagram of spectral categories
$$ 
\xymatrix{
i^{\ast}\mathbb{H}(\ca) \ar[r]^{i^{\ast}\mathbb{H}(F)} \ar[d]_{i^{\ast}\mathbb{H}(j_{\ca})} & i^{\ast}\mathbb{H}(\cb) \ar[d]^{i^{\ast}\mathbb{H}(j_{\cb})} \\
i^{\ast}\mathbb{H}(\ca_f) \ar[r]_{i^{\ast}\mathbb{H}(F_f)} & i^{\ast}\mathbb{H}(\cb_f)\,,
}
$$
with $i^{\ast}\mathbb{H}(F_f)$ a quasi-equivalence (and so a Morita equivalence). Moreover, since $j_{\ca}$ and $j_{\cb}$ are fully faithful Morita dg functors with values in Morita fibrant dg categories (see \cite[2.34]{These}), proposition \cite[4.8]{Mandell} implies that $i^{\ast}\mathbb{H}(j_{\ca})$ and $i^{\ast}\mathbb{H}(j_{\cb})$ are Morita equivalences. In conclusion, by the two out of three property, the spectral functor $i^{\ast}\mathbb{H}(F)$ is also a Morita equivalence.
\end{proof}
\begin{remark}
By remark~\ref{Morinv} and proposition~\ref{Mor-Sp}, we obtain well-defined morphisms of derivators
$$ THH, TC: \mathsf{HO}(\dgcat) \longrightarrow \mathsf{HO}(\mathsf{Sp})\,.$$ 
\end{remark}
\begin{proposition}\label{THHAdd}
The morphisms of derivators $THH$ and $TC$ are additive invariants of dg categories (see \cite[3.86]{These}).
\end{proposition}
\begin{proof}
The proof will consist on verifying that $THH$ and $TC$ satisfy all the conditions of \cite[3.86]{These}. Clearly $THH$ and $TC$ preserve the point. Moreover, by construction, the morphisms of derivators
$$i^{\ast}: \mathsf{HO}(H\mathbb{Z}\text{-}\mbox{Cat}) \longrightarrow \mathsf{HO}(\mathsf{Sp}^{\Sigma}\text{-}\mbox{Cat})$$
and
$$ THH,TC: \mathsf{HO}(\mathsf{Sp}^{\Sigma}\text{-}\mbox{Cat}) \longrightarrow \mathsf{HO}(\mathsf{Sp})$$
commute with filtered homotopy colimits. It remains to show that $THH$ and $TC$ satisfy the additivity condition A). Let
$$ 
\xymatrix{
0 \ar[r] & \mathcal{A} \ar[r]_{i_{\mathcal{A}}} & \mathcal{B} \ar@<-1ex>[l]_R
\ar[r]_P & \mathcal{C} \ar@<-1ex>[l]_{i_{\mathcal{C}}} \ar[r] & 0 \,,
}
$$
be a split short exact sequence of dg categories~\cite[3.68]{These}.
Observe that by applying the functor $i^{\ast}\mathbb{H}(-)$, we obtain a split short exact sequence of spectral categories (see section $6$ of \cite{Mandell})
$$ 
\xymatrix{
0 \ar[r] & i^{\ast}\mathbb{H}(\mathcal{A}) \ar[r] & i^{\ast}\mathbb{H}(\mathcal{B}) \ar@<-1ex>[l]
\ar[r] & i^{\ast}\mathbb{H}(\mathcal{C}) \ar@<-1ex>[l] \ar[r] & 0 \,.
}
$$
By the localization theorem \cite[6.1]{Mandell}, we obtain split triangles
$$  
\xymatrix{ 
THH(\ca) \ar[r] & THH(\cb) \ar@<-1ex>[l]
\ar[r] & THH(\cc) \ar@<-1ex>[l] \ar[r] & THH(\ca)[1]\,,
}
$$
$$  
\xymatrix{ 
TC(\ca) \ar[r] & TC(\cb) \ar@<-1ex>[l]
\ar[r] & TC(\cc) \ar@<-1ex>[l] \ar[r] & TC(\ca)[1] 
}
$$
in $\mathsf{Ho}(\mathsf{Sp})$. In conclusion, the dg functors $i_{\ca}$ and $i_{\cc}$ induce isomorphisms
$$THH(\ca) \oplus THH(\cc) \stackrel{\sim}{\longrightarrow} THH(\cb)\,,$$
$$TC(\ca) \oplus TC(\cc) \stackrel{\sim}{\longrightarrow} TC(\cb)$$
in $\mathsf{Ho}(\mathsf{Sp})$ and so the proof is finished.
\end{proof}
\section{From $K$-theory to $THH$}
In this chapter, we construct non-trivial natural transformations from the algebraic $K$-theory groups functors to the topological Hochschild homology groups functors.

\begin{notation}
We denote by
$$ K_n(-): \mathsf{Ho}(\dgcat) \longrightarrow \mbox{Mod}\text{-}\mathbb{Z}, \,\, n \geq 0$$
be the $n$th $K$-theory group functor \cite{Quillen1} \cite{Wald} and by
$$ THH_j(-): \mathsf{Ho}(\dgcat) \longrightarrow \mbox{Mod}\text{-}\mathbb{Z}, \,\, j \geq 0$$
the $j$th topological Hochschild homology group functor.
\end{notation}
\begin{theorem}\label{natural}
Suppose that we are working over the ring of integers $\mathbb{Z}$. Then we have non-trivial natural transformations
$$
\begin{array}{rcl}
 \gamma_n: K_n(-) & \Rightarrow & THH_n(-)\,,\,\,\,\,\,\, n \geq 0\\
 \gamma_{n,r}: K_n(-) & \Rightarrow & THH_{n+2r-1}(-)\,,\,\,\,\,\,\, n\geq 0\,,\ r\geq 1
\end{array}
$$
from the algebraic $K$-theory groups to the topological Hochschild homology ones.
\end{theorem}
\begin{proof}
By proposition \ref{THHAdd}, the morphism of derivators 
$$ THH:\mathsf{HO}(\dgcat) \longrightarrow \mathsf{HO}(\mathsf{Sp})$$
is an additive invariant and so by theorem~\cite[3.85]{These} it descends to the additive motivator of dg categories $\cm^{add}_{dg}$ (see \cite[3.82]{These}) and induces a triangulated functor (still denoted by $THH$)
$$ THH: \cm^{add}_{dg}(e) \longrightarrow \mathsf{Ho}(\mathsf{Sp})\,.$$
Recall that the topological Hochschild homology functor $THH_j(-), \,\,j\geq 0$ is obtained by composing $THH$ with the functor
$$ \pi^s_j: \mathsf{Ho}(\mathsf{Sp}) \longrightarrow \mbox{Mod}\text{-}\mathbb{Z}, \,\, j \geq 0\,.$$
Now, by~\cite[17.2]{Duke}, the functor
$$ K_n(-): \cm_{dg}^{add}(e) \longrightarrow \mbox{Mod}\text{-}\mathbb{Z}$$
is co-represented by $\cu_a(\mathbb{Z})[n]$. This implies, by the Yoneda lemma, that
$$ \mathsf{Nat}(K_n(-), THH_j(-)) \simeq THH_j(\cu_a(\mathbb{Z})[n])\,,$$
where $\mathsf{Nat}(-,-)$ denotes the abelian group of natural transformations.
Notice that we have the following isomorphisms
$$
THH_j(\cu_a(\mathbb{Z})[n]) \simeq THH_{j}(\cu_a(\mathbb{Z}))[n] \simeq THH_{j-n}(\cu_a(\mathbb{Z})) \simeq THH_{j-n}(\mathbb{Z})\,.
$$ 
By B{\"o}kstedt, we have the following calculation (see \cite[0.2.3]{Dundas})
\[
THH_j(\mathbb{Z})=\left\{ \begin{array}{lcl}
\mathbb{Z} & \text{if} & j=0 \\
\mathbb{Z}/r\mathbb{Z} & \text{if} & j=2r-1, \,\, r \geq 1\\
0 & & \text{otherwise}\,.
\end{array} \right.
\]
In conclusion, the canonical generators of $\mathbb{Z}$ and $\mathbb{Z}/r\mathbb{Z}$ furnishes us, by the Yoneda lemma, non-trivial natural transformations. This finishes the proof.
\end{proof}
\begin{theorem}\label{natural1}
Let $p$ be a prime number. If we work over the ring $\mathbb{Z}/p\mathbb{Z}$ we have non-trivial natural transformations
$$
\begin{array}{rcl}
 \gamma_n: K_n(-) & \Rightarrow & THH_{n+2r}(-)\,,\,\,\,\,\,\, n\,,\,r \geq 0\,.\\
\end{array}
$$
\end{theorem}
\begin{proof}
The proof is the same as the one of theorem~\ref{natural}, but over $\mathbb{Z}/p\mathbb{Z}$ we have the following calculation (see \cite[0.2.3]{Dundas})
 \[
THH_j(\mathbb{Z}/p\mathbb{Z})=\left\{ \begin{array}{lcl}
\mathbb{Z}/p\mathbb{Z} & \text{if} & j\,\,is\,\, even \\
0 & \text{if} & j\,\,is\,\,odd\,.

\end{array} \right.
\]

\end{proof}
\begin{remark}
When $X$ is a quasi-compact and quasi-separated scheme, the $THH$ of $X$ is equivalent \cite[1.3]{Mandell} to the topological Hochschild homology of schemes as defined by Geisser and Hesselholt in \cite{GH}. In particular, theorems~\ref{natural} and \ref{natural1} allows us to use all the calculations of $THH$ developed in \cite{GH}, to infer results on algebraic $K$-theory.
\end{remark}

\section{Topological equivalence theory}\label{chapter9}
In this chapter we recall Dugger-Shipley's notion of topological equivalence and show that, when $R$ is the field of rationals numbers $\mathbb{Q}$, two dg categories are topological equivalent if and only if they are quasi-equivalent (see proposition~\ref{proprational}).

Let $i: \mathbb{S} \rightarrow H\mathbb{Z}$ be the unique morphism of commutative symmetric ring spectra. By proposition~\ref{Top}, we have the restriction/extension of scalars Quillen adjunction
$$
\xymatrix{
H\mathbb{Z}\text{-}\mbox{Cat} \ar@<1ex>[d]^{i^{\ast}} \\
\mathsf{Sp}^{\Sigma}\text{-}\mbox{Cat} \ar@<1ex>[u]^{i_!}\,.
}
$$
\begin{definition}
Let $\ca$ and $\cb$ be two dg categories. We say that $\ca$ and $\cb$ are {\em topological equivalent} if the spectral categories $i^{\ast}(\mathbb{H}(\ca))$ and $i^{\ast}(\mathbb{H}(\cb))$ are isomorphic in $\mathsf{Ho}(\mathsf{Sp}^{\Sigma}\text{-}\mbox{Cat})$.
\end{definition}
\begin{remark}
If in the previous definition, $\ca$ and $\cb$ have only one object, we recover Dugger-Shipley's notion of topological equivalence for DG algebras, see \cite{DS}.
\end{remark}
\begin{remark}\label{eq->top}
The zig-zag of Quillen equivalences described in section~\ref{global} and definition~\ref{stableR} imply that:
if $\ca$ and $\cb$ are isomorphic in $\mathsf{Ho}(\dgcat)$, then they are topological equivalent.
\end{remark}
Now, suppose we work over the rationals $\mathbb{Q}$. As before, we have the restriction/extension of scalars Quillen adjunction
$$
\xymatrix{
H\mathbb{Q}\text{-}\mbox{Cat} \ar@<1ex>[d]^{i^{\ast}} \\
\mathsf{Sp}^{\Sigma}\text{-}\mbox{Cat} \ar@<1ex>[u]^{i_!}\,.
}
$$
\begin{notation}
Let $\ca$ be a spectral category. We say that $\ca$ satisfies condition $R_H)$ if the following condition is verified:
\begin{itemize}
\item[$R_H)$] for all objects $x,y \in \ca$, the symmetric spectrum $\ca(x,y)$ has rational homotopy.
\end{itemize}
\end{notation}
\begin{lemma}\label{rational}
Let $F:\ca \rightarrow \cb$ be a stable quasi-equivalence between spectral categories which satisfy condition $R_H)$. Then $i_!(F):i_!(\ca) \rightarrow i_!(\cb)$ is a stable quasi-equivalence (\ref{stableR}).
\end{lemma}
\begin{proof}
Notice that the proof of proposition~\cite[1.7]{DS} implies that, for all objects $x,y \in \ca$, the morphism
$$ i_!(\ca)(x,y)=\ca(x,y)\underset{\mathbb{S}}{\wedge}H\mathbb{Q} \longrightarrow \cb(Fx,Fy)\underset{\mathbb{S}}{\wedge}H\mathbb{Q} = i_!(\cb)(Fx,Fy)$$
is a projective stable equivalence in $H\mathbb{Q}\text{-}\mbox{Mod}$ (see~\ref{stable}).

Now, consider the commutative square
$$
\xymatrix{
\ca \ar[d]_F \ar[r]^-{\epsilon_{\ca}} & i^{\ast}i_!(\ca) \ar[d]^{i^{\ast}i_!(F)} \\
\cb \ar[r]_-{\epsilon_{\cb}} & i^{\ast}i_!(\cb)\,,
}
$$
where $\epsilon_{\ca}$ and $\epsilon_{\cb}$ are the spectral functors induced by the adjunction $(i_!, i^{\ast})$. Since $\ca$ and $\cb$ satisfy condition $R_H)$, the proof of proposition~\cite[1.7]{DS} implies that for all objects $x,y \in \ca$, the morphism
$$ \ca(x,y) \longrightarrow \ca(x,y) \underset{\mathbb{S}}{\wedge} H\mathbb{Q} = i^{\ast}i_!(\ca)(x,y)$$
is a projective stable equivalence in $H\mathbb{Q}\text{-}\mbox{Mod}$. Since $\epsilon_{\ca}$ and $\epsilon_{\cb}$ induce the identity map on objects, we conclude that $\epsilon_{\ca}$ and $\epsilon_{\cb}$ are stable quasi-equivalences and so by proposition~\ref{coincide}, that $i_!(P)$ is a stable quasi-equivalence.
\end{proof}
\begin{proposition}\label{proprational}
Let $\ca$ and $\cb$ be two dg categories over $\mathbb{Q}$. Then $\ca$ and $\cb$ are topological equivalent if and only if they are isomorphic in $\mathsf{Ho}(\dgcat_{\mathbb{Q}})$.
\end{proposition}
\begin{proof}
By remark~\ref{eq->top}, it is enough to prove one of the implications. Suppose that $\ca$ and $\cb$ are topological equivalent. Then there exists a zig-zag of stable quasi-equivalences in $\mathsf{Sp}^{\Sigma}\text{-}\mbox{Cat}$
$$ i^{\ast}(\ca) \stackrel{\sim}{\leftarrow} \cz_1 \stackrel{\sim}{\rightarrow} \ldots \stackrel{\sim}{\rightarrow} \cz_n \stackrel{\sim}{\leftarrow} i^{\ast}(\cb) $$
relating $i^{\ast}(\ca)$ and $i^{\ast}(\cb)$. By applying the functor $i_!$, we obtain by lemma~\ref{rational}, the following zig-zag of stable quasi-equivalences
$$ i_! i^{\ast}(\ca) \stackrel{\sim}{\leftarrow} i_!(\cz_1) \stackrel{\sim}{\rightarrow} \ldots \stackrel{\sim}{\rightarrow} i_!(\cz_n) \stackrel{\sim}{\leftarrow} i_! i^{\ast}(\cb)\,.$$
We now show that the $H\mathbb{Q}$-functors
$$
\begin{array}{rcl}
\eta_{\ca}: i_! i^{\ast}(\ca) \rightarrow \ca & &  \eta_{\cb}: i_!i^{\ast}(\cb) \rightarrow \cb\,,
\end{array}
$$
induced by the above adjunction $(i_!, i^{\ast})$ are in fact stable quasi-equivalences (\ref{stableR}). Notice that the proof of proposition~\cite[1.7]{DS} implies that, for all objects $x,y \in \ca$, the morphism
$$i_!i^{\ast} (\ca)(x,y) \longrightarrow \ca(x,y)$$
is a projective stable equivalence in $H\mathbb{Q}\text{-}\mbox{Mod}$. Since $\eta_{\ca}$ induces the identity map on objects, it satisfies condition $WE2')$ and so it is a stable quasi-equivalence. An analogous result holds for $\cb$ and so the proof is finished. 
\end{proof}

\appendix

\section{Adjunctions}
Let $(\cc, -\otimes-, {\bf 1}_{\cc})$ and $(\cd, -\wedge-, {\bf 1}_{\cd})$ be two symmetric monoidal categories and 
$$
\xymatrix{
\cc \ar@<1ex>[d]^N \\
\cd \ar@<1ex>[u]^L
}
$$
an adjunction, with $N$ a lax monoidal functor (\ref{lax}).
 
\begin{notation}
If $I$ is a set, we denote by $\cc^I$-$\mbox{Gr}$, resp. by $\cc^I$-$\mbox{Cat}$, the category of $\cc$-graphs with a fixed set of objects $I$, resp. the category of categories enriched over $\cc$ which have a fixed set of objects $I$. The morphisms in $\cc^I$-$\mbox{Gr}$ and $\cc^I$-$\mbox{Cat}$ induce the identity map on the objects.
\end{notation}
We have a natural adjunction
$$
\xymatrix{
\cc^I\text{-}\mbox{Cat} \ar@<1ex>[d]^U \\
\cc^I \text{-}\mbox{Gr} \ar@<1ex>[u]^{T_I},
}
$$
where  $U$ is the forgetful functor and $T_I$ is defined as
\[ T_I(\ca)(x,y):= \left\{ \begin{array}{ll}
{\bf 1}_{\cc} \amalg \underset{x, x_1, \ldots, x_n, y}{\coprod} \ca(x,x_1)  \otimes \ldots \otimes \ca(x_n,y) & \mbox{if $x=y$} \\
\underset{x, x_1, \ldots, x_n,y}{\coprod} \ca(x,x_1) \otimes \ldots \otimes \ca(x_n,y) & \mbox{if $x\neq y$} \end{array} \right. \]
Composition is given by concatenation and the unit corresponds to ${\bf 1}_{\cc}$.
\begin{remark}\label{nova1}
If the category $\cc$ carries a cofibrantly generated Quillen model structure, the categories $\cc^I$-$\mbox{Gr}$ and $\cc^I\text{-}\mbox{Cat}$ admit standard model structures. The weak equivalences (resp. fibrations) are the morphisms $F:\ca \rightarrow \cb$ such that
$$ F(x,y):\ca(x,y) \longrightarrow \cb(x,y),\,\,\,\, x,y \in I$$
is a weak equivalence (resp. fibration) in $\cc$. In fact, the Quillen model structure on $\cc$ naturally induces a model structure on $\cc^I\text{-}\mbox{Gr}$, which can be lifted along the functor $T_I$ using theorem~\cite[11.3.2]{Hirschhorn}.
\end{remark}
Clearly the adjunction $(L,N)$, induces the following adjunction
$$
\xymatrix{
\cc^I\text{-}\mbox{Gr} \ar@<1ex>[d]^N \\
\cd^I\text{-}\mbox{Gr} \ar@<1ex>[u]^L
}
$$
still denoted by $(L,N)$.

Since the functor $N: \cc \rightarrow \cd$ is lax monoidal it induces, as in \cite[3.3]{SS}, a functor
$$
\xymatrix{
\cc^I\text{-}\mbox{Cat} \ar@<1ex>[d]^{N_I} \\
\cd^I\text{-}\mbox{Cat}\,.
}
$$
More precisely, let $\cb \in \cc^I$-$\mbox{Cat}$ and $x, y$ and $z$ objects of $\cb$. Then $N_I(\cb)$ has the same objects as $\cb$, the spaces of morphisms are given by
$$ N_I(\cb)(x,y):=N\cb(x,y),\,\,\,\, x, y \in \cb$$
and the composition is defined by
$$ N\cb(x,y) \wedge N\cb(y,z) \longrightarrow N(\cb(x,y)\otimes \cb(y,z)) \stackrel{N(c)}{\longrightarrow} N\cb(x,z)\,,$$
where $c$ denotes the composition operation in $\cb$. As it is shown in section $3.3$ of \cite{SS} the functor $N_I$ admits a left adjoint $L_I$:

Let $\ca \in  \cd^I$-$\mbox{Cat}$. The value of the left adjoint $L_I$ on $\ca$ is defined as the coequalizer of two morphisms in $\cc^I$-$\mbox{Cat}$
$$ \xymatrix{ T_I L U T_IU(\ca) \ar@<1ex>[r]^{\psi_1} \ar@<-1ex>[r]_{\psi_2}  & T_I L U(\ca) \ar[r] & L_I(\ca).}$$
The morphism $\psi_1$ is obtained from the unit of the adjunction
$$ T_I U \ca \longrightarrow \ca\,,$$
by applying the composite functor $T_I L U$; the morphism $\psi_2$ is the unique morphism in $\cc^I$-$\mbox{Cat}$ induced by the $\cc^I$-$\mbox{Gr}$ morphism
$$ L U T_I U(\ca) \longrightarrow U T_I L U(\ca)$$
whose value at $L U T_I U(\ca)(x,y)\,,\,\, x, y \in I$ is given by
$$ 
\xymatrix{
\underset{x,x_1, \ldots, x_n,y}{\coprod} L(\ca(x,x_1)\otimes \ldots \otimes \ca(x_n,y)) \ar[d]^-{\phi} \\ \underset{x,x_1, \ldots, x_n,y}{\coprod} L \ca(x,x_1) \otimes \ldots \otimes L \ca(x_n,y),
}
$$
where $\phi$ is the lax {\em co}monoidal structure on $L$, induced by the lax monoidal structure on $N$, see section $3.3$ of \cite{SS}.
\begin{remark}\label{nova2}
Notice that if $\cc$ and $\cd$ are Quillen model categories and the adjunction $(L,N)$ is a weak monoidal Quillen equivalence (see \ref{smwe}), proposition~\cite[6.4]{SS} implies that we obtain a Quillen equivalence
$$
\xymatrix{
\cc^I\text{-}\mbox{Cat} \ar@<1ex>[d]^{N_I} \\
\cd^I\text{-}\mbox{Cat}  \ar@<1ex>[u]^{L_I}
}
$$
with respect to the model structure of remark~\ref{nova1}. Moreover, if $L$ is strong monoidal (\ref{lax}), the left adjoint $L_I$ is given by the original functor $L$.
\end{remark}
{\bf Left adjoint}
Notice that the functor $N_I: \cc^I$-$\mbox{Cat} \rightarrow \cd^I$-$\mbox{Cat}$, of the previous section, can be naturally defined for every set $I$ and so it induces a `global' functor
$$
\xymatrix{
\cc\text{-}\mbox{Cat} \ar@<1ex>[d]^N \\
\cd\text{-}\mbox{Cat} .
}
$$
In this section we will construct the left adjoint of $N$:

Let $\ca \in \cd\text{-}\mbox{Cat}$ and denote by $I$ its set of objects. Define $L_{cat}(\ca)$ as the $\cc$-category $L_I(\ca)$. 

Now, let $F:\ca \rightarrow \ca'$ be a $\cd$-functor. We denote by $I'$ the set of objects of $\ca'$. The $\cd$-functor $F$ induces the following diagram in $\cc\text{-}\mbox{Cat}$:
$$
\xymatrix{
T_I L U T_I U(\ca) \ar@<1ex>[r]^{\psi_1} \ar@<-1ex>[r]_{\psi_2} \ar[d] & T_I L U(\ca) \ar[d] \ar[r] & L_I(\ca)=:L_{cat}(\ca)  \\
T_{I'} L U T_{I'} U(\ca') \ar@<1ex>[r]^{\psi_1} \ar@<-1ex>[r]_{\psi_2} & T_{I'} L U(\ca') \ar[r] & L_{I'}(\ca')=:L_{cat}(\ca')\,.
}
$$
Notice that the square whose horizontal arrows are $\psi_1$ (resp. $\psi_2$) is commutative. Since the inclusions
$$ \cc^I\text{-}\mbox{Cat} \hookrightarrow \cc\text{-}\mbox{Cat} \,\,\,\,\text{and} \,\,\,\, \cc^{I'}\text{-}\mbox{Cat} \hookrightarrow \cc\text{-}\mbox{Cat}$$
clearly preserve coequalizers the previous diagram in $\cc\text{-}\mbox{Cat}$ induces a $\cc$-functor
$$ L_{cat}(F) : L_{cat}(\ca) \longrightarrow L_{cat}(\ca').$$
We have constructed a functor
$$ L_{cat}: \cd\text{-}\mbox{Cat} \longrightarrow \cc\text{-}\mbox{Cat}.$$

\begin{proposition}{\cite[5.5]{DGScat}}\label{Adjunction}
The functor $L_{cat}$ is left adjoint to $N$. 
\end{proposition}

\begin{remark}{\cite[3.3]{SS}}\label{corleftadj}
Notice that if, in the initial adjunction $(L,N)$, the left adjoint $L$ is strong monoidal (see~\ref{lax}), remark~\ref{nova2} implies that the left adjoint functor 
$$L_{cat}:\cd\text{-}\mbox{Cat} \longrightarrow \cc\text{-}\mbox{Cat}$$
is given by the original left adjoint $L$.
\end{remark}
\section{Bousfield localization techniques}
In this appendix we recall and generalize some results concerning the construction and localization of Quillen model structures. We start by stating a weaker form of the Bousfield localization theorem~\cite[X-4.1]{Jardine}.
\begin{definition}
Let $\cm$ be a Quillen model category, $Q:\cm \rightarrow \cm$ a functor and $\eta:\mbox{Id} \rightarrow Q$ a natural transformation between the identity functor and $Q$. A morphism $f:A \rightarrow B$ in $\cm$ is:
\begin{itemize}
\item[-] a {\em $Q$-weak equivalence} if $Q(f)$ is a weak equivalence in $\cm$.
\item[-] a {\em cofibration} if it is a cofibration in $\cm$.
\item[-] a {\em $Q$-fibration} if it has the R.L.P. with respect to all cofibrations which are $Q$-weak equivalences.
\end{itemize}
\end{definition}
An immediate analysis of the proof of theorem allows us to state the following general theorem~\ref{modif}. Notice that in the entire proof, we only use the right properness of $\cm$ and in the proof of lemma \cite[X-4.6]{Jardine} and theorem~\cite[X-4.8]{Jardine}, we only use condition (A3).
\begin{theorem}{\cite[X-4.1]{Jardine}}\label{modif}
Let $\cm$ be a right proper Quillen model structure, $Q:\cm \rightarrow \cm$ a functor and $\eta:\mbox{Id} \rightarrow Q$ a natural transformation such that the following three conditions hold:
\begin{itemize}
\item[(A1)] The functor $Q$ preserves weak equivalences.
\item[(A2)] The maps $\eta_{Q(A)}, Q(\eta_{A}):Q(A) \rightarrow QQ(A)$ are weak equivalences in $\cm$.
\item[(A3)] Given a diagram
$$
\xymatrix{
& B \ar[d]^p \\
A \ar[r]_{\eta_A} & Q(A)
}
$$
with $p$ a $Q$-fibration, the induced map $\eta_{{\ca}_{\ast}}:A \underset{Q(A)}{\times} B\rightarrow B$ is a $Q$-weak equivalence.
\end{itemize}
Then there is a right proper Quillen model structure on $\cm$ for which the weak equivalences are the $Q$-weak equivalences, the cofibrations those of $\cm$ and the fibrations the $Q$-fibrations.
\end{theorem}
\begin{theorem}{\cite[X-4.8]{Jardine}}\label{Qfib}
Suppose that the right proper Quillen model category $\cm$ and the functor $Q$ together satisfy the conditions for theorem~\ref{modif}. Then a map $f:A \rightarrow B$ is a $Q$-fibration if and only if it is a fibration in $\cm$ and the square
$$
\xymatrix{
A \ar[d]_f \ar[r]^{\eta_A} & Q(A) \ar[d]^{Q(f)} \\
B \ar[r]_{\eta_B} & Q(B)
}
$$
is homotopy cartesian in $\cm$.
\end{theorem}
\begin{corollary}{\cite[X-4.12]{Jardine}}\label{cormodif}
Suppose that the right proper Quillen model category $\cm$ and the functor $Q$ together satisfy the conditions for theorem~\ref{modif}. Then an object $A$ of $\cm$ is $Q$-fibrant if and only if it is fibrant in $\cm$ and the map $\eta_A:A \rightarrow Q(A)$ is a weak equivalence in $\cm$.
\end{corollary}

\section{Non-additive filtration argument}
Let $\cd$ be a closed symmetric monoidal category. We denote by $-\wedge-$ its symmetric monoidal product and  by ${\bf 1}$ its unit. We start with some general arguments.
\begin{definition}\label{funcU}
Consider the functor
$$U: \cd \longrightarrow \cd\text{-}\mbox{Cat}\,,$$
which sends an object $X\in \cd$ to the $\cd$-category $U(X)$, with two objects $1$ and $2$ and such that $U(X)(1,1)=U(X)(2,2)={\bf 1}, \,U(X)(1,2)=X$ and $U(X)(2,1)=0$, where $0$ denotes the initial object in $\cd$. Composition is naturally defined (notice that $0$ acts as a zero with respect to $\wedge$ since the bi-functor $-\wedge-$ preserves colimits in each of its variables).
\end{definition}

\begin{remark}\label{rkA}
For an object $X \in \cd$, the object $U(X) \in \cd\text{-}\mbox{Cat}$ co-represents the following functor
$$ 
\begin{array}{rcc}
\cd\text{-}\mbox{Cat} & \longrightarrow & \mathbf{Set} \\
\ca & \mapsto & \underset{(x,y) \in \ca\times \ca}{\coprod} \mathsf{Hom}_{\cd}(X,\ca(x,y))\,.
\end{array}
$$
This shows us, in particular, that the functor $U$ preserves sequentially small objects.
\end{remark}

Now, suppose that $\cd$ is a monoidal model category, with cofibrant unit, which satisfies the monoid axiom \cite[3.3]{SS}. In what follows, by {\em smash product} we mean the symmetric product $-\wedge-$ of $\cd$.
\begin{proposition}{\cite[B.3]{Spectral}}\label{clef}
Let $\ca$ be a $\cd$-category, $j:K\rightarrow L$ a trivial cofibration in $\cd$ and $U(K) \stackrel{F}{\rightarrow} \ca$ a morphism in $\cd\text{-}\mbox{Cat}$. Then, in the pushout
$$
\xymatrix{
U(K) \ar[r]^-F \ar[d]_{U(j)} \ar@{}[dr]|{\lrcorner} & \ca \ar[d]^R \\
U(L) \ar[r] & \cb
}
$$
the morphisms
$$ R(x,y):\ca(x,y) \longrightarrow \cb(x,y), \,\,\,\, x,y\in \ca$$
are weak equivalences in $\cv$.
\end{proposition}

\begin{proposition}{\cite[B.4]{Spectral}}\label{clef1}
Let $\ca$ be a $\cd$-category such that $\ca(x,y)$ is cofibrant in $\cv$ for all $x,y \in \ca$ and $i:N \rightarrow M$ a cofibration in $\cd$. Then, in a pushout in $\cd \text{-}\mbox{Cat}$
$$
\xymatrix{
U(N) \ar[r]^F \ar[d]_{U(i)} \ar@{}[dr]|{\lrcorner} & \ca \ar[d]^R \\
U(M) \ar[r] & \cb
}
$$
the morphisms
$$ R(x,y):\ca(x,y) \longrightarrow \cb(x,y), \,\,\,\, x,y\in \ca$$
are all cofibrations in $\cd$.
\end{proposition}


\begin{thebibliography}{00}

\bibitem{Borceaux} F.~Borceux, {\em Handbook of categorical
    algebra. 2}, Encyclopedia of Mathematics and its Applications,
  {\bf 51}, Cambridge University Press, 1994.

\bibitem{CN} D.~Cisinski, A.~Neeman, {\em Additivity for derivator
    $K$-theory}. Advances in Mathematics, {\bf 217} (4), 1381--1475 (2008).

\bibitem{Dugger} D.~Dugger, {\em Replacing model categories with simplicial ones}, Trans. Amer. Math. Soc. {\bf 353} (12) (2001), 5003-5027.

\bibitem{DS} D.~Dugger, B.~Shipley, {\em Topological equivalences for differential graded algebras}, Adv. Math {\bf 212}(1) (2007), 37-62.

\bibitem{Dundas} B.~Dundas, T.~Goodwillie, R.~McCarthy {\em The local structure of algebraic $K$-theory}. Available at {\tt www.math.ntnu.no/$\sim$dundas/indexeng.html}. 

 \bibitem{Drinfeld} V.~Drinfeld, {\em DG quotients of DG categories},
J. Algebra {\bf 272} (2004), 643--691.

\bibitem{Chitalk} V.~Drinfeld, {\em DG categories}. University of Chicago Geometric Langlands Seminar. Notes available at {\tt http://www.math.utexas.edu/users/benzvi/GRASP/lectures/Langlands.html}.

\bibitem{Latch} R.~Fritsch, D.~Latch, {\em Homotopy inverses for nerve}. Math. Z. {\bf 177} (1981) (2), 147-179.

\bibitem{GH} T.~Geisser, L.~Hesselholt, {\em On the $K$-theory and topological cyclic homology of smooth schemes over a discrete valuation ring}. Trans.~Amer.~Math.~Soc. {\bf 358}(1), 131--145, 2006.

\bibitem{Madsen} L.~Hesselholt, Ib.~Madsen {\em Cyclic polytopes and the $K$-theory of truncated polynomial algebras}. Invent. Math., {\bf 130}(1), 73-97, 1997.

\bibitem{Madsen1} L.~Hesselholt, Ib.~Madsen {\em On the $K$-theory of finite algebras over Witt vectors of perfect fields}. Topology, {\bf 36}(1), 29-101, 1997.

\bibitem{Madsen2} L.~Hesselholt, Ib.~Madsen {\em On the $K$-theory of local fields}. Ann. of Math., {\bf 158}(1), 1-113, 2003.

\bibitem{Hirschhorn} P.~Hirschhorn, {\em Model categories and
    their localizations}, Mathematical Surveys and Monographs, {\bf
    99}, American Mathematical Society, 2003.

\bibitem{Hovey} M.~Hovey, {\em Spectra and symmetric spectra in
    general model categories}, Journal of Pure and Applied Algebra,
  {\bf 165}, 2001, 63--127.
  
\bibitem{HoveyLivro} M.~Hovey, {\em Model categories}, Mathematical Surveys and Monographs, {\bf
    63}, American Mathematical Society, 1999.

\bibitem{HSS} M.~Hovey, B.~Shipley, J.~Smith, {\em Symmetric spectra}. J. Amer. Math. Soc. {\bf 13} (2000), no. 1, 149--208.

\bibitem{Jardine} P.~Goerss and J.~Jardine, {\em Simplicial homotopy theory}, Progress in Mathematics, {\bf 174},
  1999.

\bibitem{ICM} B.~Keller, {\em On differential graded
    categories}, International Congress of Mathematicians, Vol.~II,
  151--190, Eur.~Math.~Soc., Z{\"u}rich, 2006.

\bibitem{ENS} M.~Kontsevich, {\em  Categorification, NC Motives, Geometric Langlands and Lattice Models}. University of Chicago Geometric Langlands Seminar. Notes available at {\tt http://www.math.utexas.edu/users/benzvi/notes.html}. 

\bibitem{finMotiv} M.~Kontsevich, {\em Notes on motives in finite characteristic}. Preprint arXiv:$0702206$. To appear in Manin Festschrift.

\bibitem{Mandell} A.~Blumberg, M.~Mandell, {\em Localization theorems in topological Hochschild homology and topological cyclic homology}. Available at arXiv:$0802.3938$.

\bibitem{Quillen} D.~Quillen, {\em Homotopical algebra}, Lecture Notes
  in Mathematics, {\bf 43}, Springer-Verlag, 1967.

\bibitem{Quillen1} D.~Quillen, {\em Higher Algebraic $K$-theory, I: Higher $K$-theories}, (Proc. Conf., Battelle Memorial Inst., Seattle, Wash., 1972), 85--147. Lecture Notes in Math., Vol. {\bf 341}.

\bibitem{Schwede}  S.~Schwede, {\em An untitled book project about symmetric spectra}. Available at {\text www.math.uni-bonn.de/people/schwede}.

\bibitem{SS} S.~Schwede, B.~Shipley, {\em Algebras and modules in monoidal model categories}. Proc. London Math. Soc. (3) {\bf 80} (2000), no. 2, 491--511.

\bibitem{Shipley} B.~Shipley, {\em $H\mathbb{Z}$-algebra spectra are differential graded algebras}, Amer. J. Math, 129(2), 351--379, 2007.

\bibitem{These} G.~Tabuada, {\em Th{\'e}orie homotopique des DG-cat{\'e}gories}, author's Ph.D thesis. Available at arXiv:$0710.4303$. 

\bibitem{Duke} G.~Tabuada, {\em Higher $K$-theory via universal invariants}. Available at arXiv:$0706.2420$. To appear in Duke Math. Journal.

\bibitem{Spectral} G.~Tabuada, {\em Homotopy theory of Spectral categories}. Available at arXiv:$0801.4524$.

\bibitem{DGScat} G.~Tabuada, {\em Differential graded versus Simplicial categories}. Available at arXiv:$0711.3845$.

\bibitem{Wald} F.~Waldhausen, {\em Algebraic K-theory of spaces},
  Algebraic and geometric topology (New Brunswick, N.~J., 1983),
  318--419, Lecture Notes in Math., 1126, Springer, Berlin, 1985.

\end{thebibliography}
\end{document}